\documentclass[11pt]{article}
\usepackage{amssymb,amsmath}
\usepackage[mathscr]{eucal}
\usepackage[cm]{fullpage}
\usepackage[english]{babel}
\usepackage[latin1]{inputenc}
\usepackage{color}

\def\dom{\mathop{\mathrm{Dom}}\nolimits}
\def\im{\mathop{\mathrm{Im}}\nolimits}
\def\d{\mathrm{d}}
\def\N{\mathbb N}
\def\PT{\mathcal{PT}}
\def\T{\mathcal{T}}
\def\Sym{\mathcal{S}}

\def\DP{\mathcal{DP}}
\def\DPS{\mathcal{DPS}}
\def\ODP{\mathcal{ODP}}

\def\I{\mathcal{I}}
\newtheorem{theorem}{Theorem}[section]
\newtheorem{proposition}[theorem]{Proposition}
\newtheorem{corollary}[theorem]{Corollary}
\newtheorem{lemma}[theorem]{Lemma}

\newenvironment{proof}{\begin{trivlist}\item[\hskip%
\labelsep{\bf Proof.}]}%
{\qed\rm\end{trivlist}}
\newcommand{\qed}{{\unskip\nobreak
\hfil\penalty50\hskip .001pt \hbox{}
          \nobreak\hfil
         \vrule height 1.2ex width 1.1ex depth -.1ex
           \parfillskip=0pt\finalhyphendemerits=0\medbreak}}
\newcommand{\lastpage}{\addresss}

\newcommand{\addresss}{\small \sf  
\noindent{\sc V\'\i tor H. Fernandes}, 
Center for Mathematics and Applications (CMA), 
FCT NOVA and Department of Mathematics, FCT NOVA, 
Faculdade de Ci\^encias e Tecnologia, 
Universidade Nova de Lisboa, 
Monte da Caparica, 
2829-516 Caparica, 
Portugal; 
e-mail: vhf@fct.unl.pt. 

\medskip

\noindent{\sc T\^ania Paulista}, 
Departamento de Matem\'atica, 
Faculdade de Ci\^encias e Tecnologia, 
Universidade NOVA de Lisboa, 
Monte da Caparica, 
2829-516 Caparica, 
Portugal; 
e-mail: t.paulista@campus.fct.unl.pt. 
}

\title{On the monoid of partial isometries of a finite star graph}

\author{V\'\i tor H. Fernandes\footnote{This work is funded by national funds through the FCT - Funda\c c\~ao para a Ci\^encia e a Tecnologia, I.P., under the scope of the projects UIDB/00297/2020 and UIDP/00297/2020 (Center for Mathematics and Applications).}~
and T\^ania Paulista 
}

\begin{document}

\maketitle

\begin{abstract}
In this paper we consider the monoid $\DPS_n$ of all partial isometries of a star graph $S_n$ with $n$ vertices. 
Our main objectives are to determine the rank and to exhibit a presentation of $\DPS_n$. 
We also describe Green's relations of $\DPS_n$ and calculate its cardinal.
\end{abstract}

\medskip



\section{Introduction and Preliminaries}\label{presection} 

Let $\Omega$ be a finite set. We denote by $\PT(\Omega)$ the monoid (under composition) of all 
partial transformations on $\Omega$, by $\T(\Omega)$ the submonoid of $\PT(\Omega)$ of all 
full transformations on $\Omega$, by $\I(\Omega)$ 
the \textit{symmetric inverse monoid} on $\Omega$, i.e. 
the inverse submonoid of $\PT(\Omega)$ of all 
partial permutations on $\Omega$, 
and by $\Sym(\Omega)$ the \textit{symmetric group} on $\Omega$, 
i.e. the subgroup of $\PT(\Omega)$ of all 
permutations on $\Omega$. 

\smallskip 

Recall that the \textit{rank} of a (finite) monoid $M$ is the minimum size of a generating set of $M$, i.e. 
the minimum of the set $\{|X|\mid \mbox{$X\subseteq M$ and $X$ generates $M$}\}$. 

Let $\Omega$ be a finite set
with at least $3$ elements.
It is well-known that $\Sym(\Omega)$
has rank $2$ (as a semigroup, a monoid or a group) and 
$\T(\Omega)$, $\I(\Omega)$ and $\PT(\Omega)$ have
ranks $3$, $3$ and $4$, respectively.
The survey \cite{Fernandes:2002survey} presents 
these results and similar ones for other classes of transformation monoids,
in particular, for monoids of order-preserving transformations and
for some of their extensions. 
For example, the rank of the extensively studied monoid of all order-preserving transformations of a $n$-chain is $n$,  
which was proved by Gomes and Howie \cite{Gomes&Howie:1992} in 1992. 
More recently, for instance, the papers 
\cite{
Araujo&al:2015,
Fernandes&al:2014,
Fernandes&al:2019,
Fernandes&Quinteiro:2014,
Fernandes&Sanwong:2014} 
are dedicated to the computation of the ranks of certain classes of transformation semigroups or monoids.

\smallskip 

A \textit{monoid presentation} is an ordered pair 
$\langle A\mid R\rangle$, where $A$ is a set, often called an \textit{alphabet}, 
and $R\subseteq A^*\times A^*$ is a set of relations of 
the free monoid $A^*$ generated by $A$. 
A monoid $M$ is said to be 
\textit{defined by a presentation} $\langle A\mid R\rangle$ if $M$ is
isomorphic to $A^*/\rho_R$, where $\rho_R$ denotes the smallest
congruence on $A^*$ containing $R$. 

Given a finite monoid, it is clear that we can always exhibit
a presentation for it, at worst by enumerating all elements from its multiplication table,
but clearly this is of no interest, in general. So, by determining a
presentation for a finite monoid, we mean to find in some sense a
\textit{nice} presentation (e.g. with a small number of generators and
relations).

A presentation for the symmetric group $\Sym(\Omega)$ was determined by Moore \cite{Moore:1897} over a century ago (1897). 
For the full transformation monoid $\T(\Omega)$, a presentation  
was given in 1958 by A\u{\i}zen\v{s}tat \cite{Aizenstat:1958} in terms of a certain 
type of two generator presentation for the symmetric group $\Sym(\Omega)$, 
plus an extra generator and seven more relations. 
Presentations for the partial transformation monoid $\PT(\Omega)$ 
and for the symmetric inverse monoid $\I(\Omega)$
were found by Popova \cite{Popova:1961} in 1961. 
In 1962, A\u{\i}zen\v{s}tat \cite{Aizenstat:1962} and Popova \cite{Popova:1962} exhibited presentations for the monoids of 
all order-preserving transformations and of all order-preserving partial transformations of a finite chain, respectively, and from the sixties until our days several authors obtained presentations for many classes of monoids. 
See also \cite{Ruskuc:1995}, the survey \cite{Fernandes:2002survey} and, 
for example, 
\cite{Cicalo&al:2015,
East:2011, 
Feng&al:2019,
Fernandes:2001, 
Fernandes&Gomes&Jesus:2004, 
Fernandes&Quinteiro:2016, 
Howie&Ruskuc:1995}. 

\medskip 

Now, let $G=(V,E)$ be a finite simple connected graph. 

The (\textit{geodesic}) \textit{distance} between two vertices $x$ and $y$ of $G$, denoted by $\d_G(x,y)$, is the length of a shortest path between $x$ and $y$, i.e. the number of edges in a shortest path between $x$ and $y$. 

Let $\alpha\in\PT(V)$. We say that $\alpha$ is a \textit{partial isometry} or \textit{distance preserving partial transformation} of $G$ if 
$$
\d_G(x\alpha,y\alpha) = \d_G(x,y) ,
$$
for all $x,y\in\dom(\alpha)$. Denote by $\DP(G)$ the subset of $\PT(V)$ of all partial isometries of $G$. Clearly, $\DP(G)$ is a submonoid of $\PT(V)$. Moreover, as a consequence of the property 
$$
\d_G(x,y)=0 \quad \text{if and only if} \quad x=y, 
$$
for all $x,y\in V$, it immediately follows that $\DP(G)\subseteq\I(V)$. In fact, furthermore, we have:

\begin{proposition} \label{injective} 
Let $G=(V,E)$ be a finite simple connected graph. 
Then $\DP(G)$ is an inverse submonoid of $\I(V)$.
\end{proposition}
\begin{proof}
It suffices to show that, for $\alpha\in\DP(G)$,
the inverse transformation $\alpha^{-1}$ (i.e. the inverse of the element $\alpha$ in $\I(V)$) is also a partial isometry of $G$. 

Let $\alpha\in\DP(G)$ and take $x,y\in\dom(\alpha^{-1})$. Then $x\alpha^{-1},y\alpha^{-1}\in\dom(\alpha)$ and so 
$$
\d_G(x\alpha^{-1},y\alpha^{-1}) = \d_G((x\alpha^{-1})\alpha,(y\alpha^{-1})\alpha)=\d_G(x,y),
$$
as required. 
\end{proof} 

Observe that, if $G=(V,E)$ is a complete graph, i.e. $E=\{\{x,y\}\mid x,y\in V, x\neq y\}$, then $\DP(G)=\I(V)$. 

On the other hand, for $n\in\N$, consider the undirected path $P_n$ with $n$ vertices, i.e. 
$$
P_n=\left(\{1,\ldots,n\},\{\{i,i+1\}\mid i=1,\ldots,n-1\}\right).
$$
Then, obviously, $\DP(P_n)$ coincides with the monoid 
$$
\DP_n=\{\alpha\in\I(\{1,2,\ldots,n\} \mid |i\alpha-j\alpha|=|i-j|, \mbox{for all $i,j\in\dom(\alpha)$}\}
$$
of all partial isometries on $\{1,2,\ldots,n\}$. 

The study of partial isometries on $\{1,2,\ldots,n\}$ was initiated 
by Al-Kharousi et al.~\cite{AlKharousi&Kehinde&Umar:2014,AlKharousi&Kehinde&Umar:2016}. 
The first of these two papers is dedicated to investigating some combinatorial properties of 
the monoid $\DP_n$ and of its submonoid $\ODP_n$ of all order-preserving (considering the usual order of $\N$) partial isometries, in particular, their cardinalities. The second paper presents the study of some of their algebraic properties, namely Green's structure and ranks. Presentations for both the monoids $\DP_n$ and $\ODP_n$ were given by the first author and Quinteiro in \cite{Fernandes&Quinteiro:2016}. 

\smallskip 

Now, for $n\in\N$, consider the \textit{star graph} 
$$
S_n=(\{0,1,\ldots, n-1\}, \{\{0,i\}\mid i=1,\ldots,n-1\}) 
$$
with $n$ vertices.  
These very elementary graphs, which are a particular kind of \textit{trees} and also of \textit{complete bipartite graphs}, play a significant role in Graph Theory. For example, through the notions of \textit{star chromatic number} or \textit{star arboricity}. We may also find important applications of star graphs in Computer Science, in particular, in Distributed Computing the \textit{star network} is one of the most common computer network topologies. 

\smallskip 

This paper is devoted to studying the monoid $\mathcal{DP}(S_n)$ of all partial isometries of $S_n$, 
which from now on we denote simply by $\DPS_n$. 
Since we are considering $\{0,1,\ldots, n-1\}$ as the set of vertices of the star $S_n$, then $\DPS_n$ is an inverse submonoid of the symmetric inverse monoid $\I({\{0,1,\ldots,n-1\}})$. 

\smallskip 

In Section \ref{basic properties}, we present some basic properties of $\DPS_n$; in particular, we calculate de cardinal of $\DPS_n$ and describe its Green's relations. 
Section \ref{genrankdps} presents a generating set of $\DPS_n$ of a minimum size, which provides the rank of $\DPS_n$.
Finally, in Section \ref{dpspres} we determine a presentation for the monoid $\DPS_n$. 

\smallskip 

Throughout this paper we will consider $\PT(Y)\subseteq\PT(X)$, 
whenever $X$ and $Y$ are sets such that $Y\subseteq X$. 

\smallskip 

For general background on Semigroup Theory and standard notations, we refer the reader to Howie's book \cite{Howie:1995}.

\smallskip 

We would like to point out that we made considerable use of computational tools, namely GAP \cite{GAP4}.

\section{Basic Properties of $\DPS_n$}
\label{basic properties} 

Let $n\in\N$. 

We start this section by observing that, clearly, 
$$ 
\d_{S_n}(x,y)=\left\{\begin{array}{ll} 
0& \mbox{if $x=y$}\\
1& \mbox{if $x\neq y$ and either $x=0$ or $y=0$} \\
2& \mbox{if $x\neq y$ and $x\neq 0$ and $y\neq 0$}
\end{array}\right.
$$ 
for all $x,y\in\{0,1,\ldots, n-1\}$. 

\medskip 

Observe that 
$$
\DPS_1=\left\{\emptyset, \begin{pmatrix}0\\0\end{pmatrix}\right\} = \I(\{0\})
$$
and
$$
\DPS_2=\left\{\emptyset,\begin{pmatrix}0\\0\end{pmatrix},
\begin{pmatrix}0\\1\end{pmatrix},\begin{pmatrix}1\\0\end{pmatrix},
\begin{pmatrix}1\\1\end{pmatrix},\begin{pmatrix}0&1\\0&1\end{pmatrix},
\begin{pmatrix}0&1\\1&0\end{pmatrix}\right\} = \I(\{0,1\}). 
$$
On the other hand, for $n\geqslant 3$, for example 
$\begin{pmatrix}0&1&2\\1&0&2\end{pmatrix}\not\in\DPS_n$ and so 
$\DPS_n$ is a proper inverse submonoid of the symmetric inverse monoid $\I({\{0,1,\ldots,n-1\}})$. 
Moreover, we have the following description of the elements of $\DPS_n$, which is a routine matter to prove. 

\begin{proposition} \label{dpschar} 
 Let $\alpha \in \PT(\{0,1,\ldots,n-1\})$. 
\begin{enumerate}
   \item If $|\dom(\alpha)|\leqslant 1$ then $\alpha\in\DPS_n$. 
    \item If $|\dom(\alpha)|\geqslant 2$ and $0 \notin \dom(\alpha)$ then the following statements are equivalent:
    \begin{enumerate}
        \item $\alpha \in \DPS_n$;
        \item $\alpha$ is injective and $0\notin \im(\alpha)$;
        \item $\alpha \in \I(\{1,2,\ldots,n-1\})$.
    \end{enumerate}
    \item If $|\dom(\alpha)|=2$ and $0 \in \dom(\alpha)$ then the following statements are equivalent:
    \begin{enumerate}
        \item $\alpha \in \DPS_n$;
        \item $\alpha$ is injective and $0\in \im(\alpha)$.
    \end{enumerate}
    \item If $|\dom(\alpha)|\geqslant 3$ and $0 \in \dom(\alpha)$ then the following statements are equivalent:
    \begin{enumerate}
        \item $\alpha \in \DPS_n$;
        \item $\alpha$ is injective and $0\alpha=0$;
        \item $\alpha|_{\dom(\alpha)\setminus\{0\}}\in \I(\{1,2,\ldots,n-1\})$ and $0\alpha=0$. 
       \end{enumerate}
\end{enumerate}
\end{proposition}

\medskip 

Let $n\geqslant 2$. For each $\xi\in \I(\{1,2,\ldots,n-1\})$, define $\bar\xi\in \I(\{0,1,2,\ldots,n-1\})$ by $0\bar\xi=0$ and $\bar\xi|_{\dom(\xi)}=\xi$. 
Taking into account the previous proposition, it is easy to conclude that 
\begin{equation}\label{dpsn}
\begin{array}{rcl}
\DPS_n&=&\{\bar\xi\in \I(\{0,1,2,\ldots,n-1\})\mid \xi\in \I(\{1,2,\ldots,n-1\})\} \cup 
\I(\{1,2,\ldots,n-1\}) \\
&& \cup 
\left\{\begin{pmatrix}0&i\\j&0\end{pmatrix}\mid 1\leqslant i,j\leqslant n-1\right\}\cup 
\left\{\begin{pmatrix}0\\i\end{pmatrix},\begin{pmatrix}i\\0\end{pmatrix} \mid 1\leqslant i\leqslant n-1\right\}.
\end{array}
\end{equation} 
Notice that, all these four subsets of $\DPS_n$ are pairwise disjoint. 

Clearly, 
$$
\{\alpha\in\DPS_n\mid |\dom(\alpha)|=n\}
$$ 
is the group of units of $\DPS_n$ and 
$$
\{\alpha\in\DPS_n\mid \mbox{$0\in\dom(\alpha)$ and $0\alpha=0$}\} = \{\bar\xi\in \I(\{0,1,2,\ldots,n-1\})\mid \xi\in \I(\{1,2,\ldots,n-1\})\} 
$$ 
is a submonoid of $\DPS_n$.  Moreover, for $n\geqslant 3$, this last submonoid of $\DPS_n$ has the same group of units as $\DPS_n$ 
and we have the following relations with 
the symmetric group $\mathcal{S}(\{1,2,\ldots,n-1\})$ of degree $n-1$ 
and the symmetric inverse monoid $\I(\{1,2,\ldots,n-1\})$. 

\begin{proposition}\label{subiso}
For $n\geqslant 3$, the group of units 
$\{\alpha\in\DPS_n\mid |\dom(\alpha)|=n\}$ 
of $\DPS_n$ is isomorphic to the symmetric group 
$\mathcal{S}(\{1,2,\ldots,n-1\})$ of order $n-1$. 
Moreover, 
$\{\alpha\in\DPS_n\mid \mbox{$0\in\dom(\alpha)$ and $0\alpha=0$}\}$ is a submonoid of $\DPS_n$ 
isomorphic to the symmetric inverse monoid 
$\I(\{1,2,\ldots,n-1\})$. 
\end{proposition} 
\begin{proof}
Consider the mapping $\psi: \I(\{1,2,\ldots,n-1\}) \longrightarrow\I(\{0,1,2,\ldots,n-1\})$ defined by $\xi\psi=\bar\xi$, for all $\xi\in\I(\{1,2,\ldots,n-1\})$. 
It is easy to check that $\psi$ is an injective homomorphism of monoids and, clearly, 
we have 
$$
\I(\{1,2,\ldots,n-1\})\psi=\{\alpha\in\DPS_n\mid \mbox{$0\in\dom(\alpha)$ and $0\alpha=0$}\}
$$
and 
$$
\mathcal{S}(\{1,2,\ldots,n-1\})\psi=\{\alpha\in\DPS_n\mid |\dom(\alpha)|=n\},  
$$
which proves the result. 
\end{proof} 

\medskip 

Next, recall that, for a finite set $\Omega$ with $n\in\N$ elements, it is well known that 
the size of the symmetric inverse monoid $\I(\Omega)$ is 
$$
|\I(\Omega)| = \sum_{k=0}^{n} \binom{n}{k}^2k! 
$$
(see \cite{Fernandes:2002survey,Howie:1995}). Therefore, from (\ref{dpsn}) and Proposition \ref{subiso}, we have 
$$
\begin{array}{rcl}
|\DPS_n|&=&|\{\bar\xi\in \I(\{0,1,2,\ldots,n-1\})\mid \xi\in \I(\{1,2,\ldots,n-1\})\}| +  
|\I(\{1,2,\ldots,n-1\})| \\
&& +
\left|\left\{\begin{pmatrix}0&i\\j&0\end{pmatrix}\mid 1\leqslant i,j\leqslant n-1\right\}\right| +  
\left|\left\{\begin{pmatrix}0\\i\end{pmatrix},\begin{pmatrix}i\\0\end{pmatrix} \mid 1\leqslant i\leqslant n-1\right\}\right| \\
& = &\displaystyle \sum_{k=0}^{n-1} \binom{n-1}{k}^2k! + \sum_{k=0}^{n-1} \binom{n-1}{k}^2k! + (n-1)^2 + 2(n-1) \\
& = &\displaystyle 1+n^2+ 2 \sum_{k=1}^{n-1} \binom{n-1}{k}^2k!
\end{array}
$$
and so we have proved: 

\begin{theorem}
For all $n\in\N$, 
$$
|\DPS_n|=1+n^2+ 2 \sum_{k=1}^{n-1} \binom{n-1}{k}^2k!.
$$ 
\end{theorem} 

The table below gives us an idea of the size of $\DPS_n$. 
\begin{center}
$\begin{array}{|c|c|}\cline{1-2}
n & |\DPS_n|  \\ \cline{1-2}  
1 & 2 \\ \cline{1-2}
2 &  7 \\ \cline{1-2}
3 &  22 \\ \cline{1-2}
4 &  83 \\ \cline{1-2}
5  &  442 \\ \cline{1-2}
6  &  3127 \\ \cline{1-2}
7  &  26702 \\ \cline{1-2}
8  &  261907 \\ \cline{1-2}
9  &  2883538 \\ \cline{1-2}
10  &  35144327 \\ \cline{1-2}
\end{array}$ \quad
$\begin{array}{|c|c|}\cline{1-2}
n &  |\DPS_n| \\ \cline{1-2}
11 & 469324582 \\ \cline{1-2}
12 &  6810715507 \\ \cline{1-2}
13 &  106668909002  \\ \cline{1-2}
14 &  1792648617463 \\ \cline{1-2}
15 &  32167115690782 \\ \cline{1-2}
16 &  613654341732467 \\ \cline{1-2}
17 &  12399337905055522 \\ \cline{1-2}
18 &   264481977288432007 \\ \cline{1-2}
19 &   5937942527822578358 \\ \cline{1-2}
20 &  139949655415806098707 \\ \cline{1-2}
\end{array}$
\end{center}

\medskip 

In the rest of this section we will describe Green's relations of $\DPS_n$. 
Remember that, given a set $\Omega$ and an inverse submonoid $M$ of $\I(\Omega)$, it is well known that 
the Green's relations $\mathscr{L}$, $\mathscr{R}$ and $\mathscr{H}$
of $M$ can be described as following: for $\alpha, \beta \in M$,
\begin{itemize}
\item $\alpha \mathscr{L} \beta$ if and only if $\im(\alpha) = \im(\beta)$;

\item $\alpha \mathscr{R} \beta$ if and only if $\dom(\alpha) = \dom(\beta)$;

\item $\alpha \mathscr{H} \beta $ if and only if $\im(\alpha) = \im(\beta)$ and $\dom(\alpha) = \dom(\beta)$.
\end{itemize}
In $\I(\Omega)$ we also have 
\begin{itemize}
\item $\alpha \mathscr{J} \beta$ if and only if $|\dom(\alpha)| = |\dom(\beta)|$ (if and only if $|\im(\alpha)| = |\im(\beta)|$). 
\end{itemize}

Since $\DPS_n$ is an inverse submonoid of $\I(\{0,1,2,\ldots,n-1\})$, 
it remains to find a description of its Green's relation $\mathscr{J}$. 
Recall that, for a finite monoid, we have $\mathscr{J}=\mathscr{D} \;(=\mathscr{L}\circ\mathscr{R}=\mathscr{R}\circ\mathscr{L}$). 

\begin{theorem} 
Let $\alpha, \beta \in \DPS_n$. Then $\alpha \mathscr{J} \beta$ if and only if one of the following properties is satisfied:
\begin{enumerate}
\item $|\dom(\alpha)|=|\dom(\beta)|=1$;
\item $|\dom(\alpha)|=|\dom(\beta)|$ and $0\notin\dom(\alpha)\cup\dom(\beta)$;
\item $|\dom(\alpha)|=|\dom(\beta)|$ and $0\in\dom(\alpha)\cap\im(\beta)$.
\end{enumerate}
\end{theorem}

\begin{proof} 
We begin by supposing that $\alpha \mathscr{J} \beta$. Then, as remembered above, 
we have $|\dom(\alpha)| = |\dom(\beta)|=|\im(\alpha)| = |\im(\beta)|$. 
On the other hand, since $\mathscr{J}=\mathscr{D}=\mathscr{R}\circ\mathscr{L}$, 
then there exists $\zeta\in\DPS_n$ such that $\alpha\mathscr{R}\zeta$ and $\zeta\mathscr{L}\beta$. Consequently, 
we have $\dom(\alpha)=\dom(\zeta)$ and $\im(\zeta)=\im(\beta)$. 

If $|\dom(\alpha)|=|\dom(\beta)|=0$ then $\alpha=\beta=\emptyset$ and so $0\not\in\dom(\alpha)\cup\dom(\beta)$. 
Hence, Property 2 is satisfied.

If $|\dom(\alpha)|=|\dom(\beta)|=1$ then it is immediate that Property 1 is satisfied.

Next, suppose that $|\dom(\alpha)|=|\dom(\beta)|\geqslant 2$ and $0\notin\dom(\alpha)$. Then $0\not\in\dom(\zeta)$ and so, by Proposition \ref{dpschar}, $0\notin\im(\zeta)=\im(\beta)$, whence $0\notin\dom(\beta)$, 
again by Proposition \ref{dpschar}. Hence, Property 2 is satisfied.

Now, suppose that $|\dom(\alpha)|=|\dom(\beta)| = 2$ and $0\in\dom(\alpha)$. Then $0\in\dom(\zeta)$, which implies, by Proposition \ref{dpschar}, that $0\in\im(\zeta)=\im(\beta)$ and so, again by Proposition \ref{dpschar}, we have  $0\in\dom(\beta)$. Hence, Property 3 holds.

Finally, suppose that $|\dom(\alpha)|=|\dom(\beta)|\geqslant 3$ and $0\in\dom(\alpha)=\dom(\zeta)$. Then, by Proposition \ref{dpschar}, we conclude that $0\zeta=0$, whence $0\in\im(\zeta)=\im(\beta)$. Then, once again by Proposition \ref{dpschar}, we deduce that $0\in\dom(\beta)$. Hence, Property 3 holds.

\smallskip 

We now prove the converse implication.

\smallskip 

First, suppose that Property 1 is verified, i.e. $|\dom(\alpha)|=|\dom(\beta)|=1$. 
Let $i,j\in\{0,1,\ldots,n-1\}$ be such that $\dom(\alpha)=\{i\}$ and $\dom(\beta)=\{j\}$. 
Then
$$
\zeta_1=\begin{pmatrix}i\\j\end{pmatrix}, \quad \zeta_2=\begin{pmatrix}j\beta\\i\alpha\end{pmatrix}, \quad \zeta_3=\begin{pmatrix}j\\i\end{pmatrix} \quad\textrm{and}\quad \zeta_4=\begin{pmatrix}i\alpha\\j\beta\end{pmatrix}
$$
are isometries of $S_n$ and, clearly, $\alpha=\zeta_1\beta\zeta_2$ and $\beta=\zeta_3\alpha\zeta_4$, 
whence $\alpha\mathscr{J}\beta$.

\smallskip 

Next, we suppose that Property 2 holds, i.e. $|\dom(\alpha)|=|\dom(\beta)|$ and $0\notin\dom(\alpha)\cup\dom(\beta)$.

If $|\dom(\alpha)|=|\dom(\beta)|=0$ then $\alpha=\beta=\emptyset$, whence $\alpha\mathscr{J}\beta$.

If $|\dom(\alpha)|=|\dom(\beta)|=1$ then Property 1 is also verified and, as proved above, we have $\alpha\mathscr{J}\beta$.

Now, assume that $|\dom(\alpha)|=|\dom(\beta)|\geqslant 2$. 
Then, by Proposition \ref{dpschar}, we get $\alpha,\beta \in \I(\{1,2,\ldots,n-1\})$. 
Hence, from $|\dom(\alpha)|=|\dom(\beta)|$, we obtain $\alpha\mathscr{J}\beta$ in $\I(\{1,2,\ldots,n-1\})$ and thus 
we also have $\alpha\mathscr{J}\beta$ in $\DPS_n$. 

\smallskip 

Finally, suppose Property 3 is verified, i.e. $|\dom(\alpha)|=|\dom(\beta)|$ and $0\in\dom(\alpha)\cap\dom(\beta)$.

If $|\dom(\alpha)|=|\dom(\beta)|=1$ then Property 1 is also verified and, again as proved above, we have $\alpha\mathscr{J}\beta$.

Next, assume that $|\dom(\alpha)|=|\dom(\beta)|=2$. 
Then, by Proposition \ref{dpschar}, we also have $0\in\im(\alpha)\cap\im(\beta)$.

Let $i,j\in\{1,\ldots,n-1\}$ be such that $\dom(\alpha)=\{0,i\}$ and $\dom(\beta)=\{0,j\}$. 

Define $\zeta_1,\zeta_2\in\mathcal{PT}_n$ by 
$$
\dom(\zeta_1)=\dom(\alpha), \quad 0\zeta_1=0 \quad\textrm{and}\quad i\zeta_1=j
$$
and 
$$
\dom(\zeta_2)=\im(\beta), \quad (0\beta)\zeta_2=0\alpha \quad\textrm{and}\quad (j\beta)\zeta_2=i\alpha.
$$
It is easy to conclude that $\zeta_1,\zeta_2\in\DPS_n$, $\alpha=\zeta_1\beta\zeta_2$ and  
$\beta=\zeta_1^{-1}\alpha\zeta_2^{-1}$, whence $\alpha\mathscr{J}\beta$.

Finally, consider that $|\dom(\alpha)|=|\dom(\beta)|\geqslant 3$. Then, by Proposition \ref{dpschar}, we have 
$$
\alpha|_{\dom(\alpha)\setminus\{0\}},\beta|_{\dom(\beta)\setminus\{0\}}\in \I(\{1,2,\ldots,n-1\})
$$ 
and $0\alpha=0\beta=0$. 
As $|\dom(\alpha|_{\dom(\alpha)\setminus\{0\}})|=|\dom(\beta|_{\dom(\beta)\setminus\{0\}})|$, we get 
$\alpha|_{\dom(\alpha)\setminus\{0\}}\,\mathscr{J}\,\beta|_{\dom(\beta)\setminus\{0\}}$ in $\I(\{1,2,\ldots,n-1\})$ 
and so, in view of the proof of Proposition \ref{subiso}, we conclude that 
$$
\alpha =(\alpha|_{\dom(\alpha)\setminus\{0\}})\psi\,\mathscr{J}\,(\beta|_{\dom(\beta)\setminus\{0\}})\psi=\beta
$$ 
in $\DPS_n$, as required. 
\end{proof}

\section{Generators and Rank of $\DPS_n$}
\label{genrankdps}

Let $\Omega$ be a finite set with $n\in\N$ elements. It is well known that the symmetric inverse monoid $\I(\Omega)$ is 
generated by its group of units, i.e. the symmetric group $\mathcal{S}(\Omega)$ of degree $n$, and any one transformation of rank $n-1$ 
(see \cite{Fernandes:2002survey,Howie:1995}). For instance, for $n\geqslant2$, the symmetric inverse monoid $\I(\{1,2,\ldots,n\})$ 
is generated by the following transformations 
$$
\begin{pmatrix}1&2&\cdots&n-1&n\\2&3&\cdots&n&1\end{pmatrix}, \quad 
\begin{pmatrix}1&2&3&\cdots&n\\2&1&3&\cdots&n\end{pmatrix}\quad\text{and}\quad 
\begin{pmatrix}1&2&\cdots&n-1\\1&2&\cdots&n-1\end{pmatrix}. 
$$ 
Notice that, in particular for $n=2$, the first two previous transformations coincide and so we simply obtain   
$\I(\{1,2\})=\left\langle\begin{pmatrix}1&2\\2&1\end{pmatrix},\begin{pmatrix}1\\1\end{pmatrix}\right\rangle$. 

\medskip

In this section we will show that $\DPS_n$ has rank $3$, for $n=3$, and rank $5$, for $n\geqslant4$, by exhibiting a set of generators with a minimum number of elements. Recall that, we already observed that $\DPS_1=\I(\{0\})$ and $\DPS_2=\I(\{0,1\})$, which are monoids with ranks $1$ and $2$, respectively. 

\medskip 

Let $n\geqslant 3$ and consider the following partial isometries of $S_n$:
$$
\alpha_1=\begin{pmatrix}0&1&2&\cdots&n-2&n-1\\0&2&3&\cdots&n-1&1\end{pmatrix}, \quad 
\alpha_2=\begin{pmatrix}0&1&2&3&\cdots&n-1\\0&2&1&3&\cdots&n-1\end{pmatrix}, \quad
\beta_1=\begin{pmatrix}0&1&\cdots&n-2\\0&1&\cdots&n-2\end{pmatrix}, 
$$
$$
\beta_2=\begin{pmatrix}1&2&\cdots&n-1\\1&2&\cdots&n-1\end{pmatrix} \quad\text{and}\quad 
\gamma=\begin{pmatrix}0&1\\1&0\end{pmatrix}.
$$

Next, we show that these transformations generate $\DPS_n$. 

\begin{proposition}\label{dpsgen}
Let $n\geqslant 3$. Then  $\DPS_n=\langle\alpha_1,\alpha_2,\beta_1,\beta_2,\gamma\rangle$. 
Moreover, in particular, $\DPS_3=\langle\alpha_1,\beta_2,\gamma\rangle$.
\end{proposition}

\begin{proof}
First of all, notice that 
$\alpha_1,\alpha_2,\beta_1\in 
\{\alpha\in\DPS_n\mid \mbox{$0\in\dom(\alpha)$ and $0\alpha=0$}\} = \I(\{1,2,\ldots,n-1\})\psi$, 
where $\psi$ is the injective homomorphism of monoids defined in the proof of Proposition \ref{subiso}, 
$$
\alpha_1\beta_2=\begin{pmatrix}1&2&\cdots&n-2&n-1\\2&3&\cdots&n-1&1\end{pmatrix}, \quad 
\alpha_2\beta_2=\begin{pmatrix}1&2&3&\cdots&n-1\\2&1&3&\cdots&n-1\end{pmatrix}, \quad
\beta_1\beta_2=\begin{pmatrix}1&\cdots&n-2\\1&\cdots&n-2\end{pmatrix}, 
$$
$(\alpha_1\beta_2)\psi=\alpha_1$, $(\alpha_2\beta_2)\psi=\alpha_2$ and $(\beta_1\beta_2)\psi=\beta_1$, 
whence $\langle\alpha_1\beta_2,\alpha_2\beta_2,\beta_1\beta_2\rangle=\I(\{1,2,\ldots,n-1\})$ 
and $\langle\alpha_1,\alpha_2,\beta_1\rangle=\{\alpha\in\DPS_n\mid \mbox{$0\in\dom(\alpha)$ and $0\alpha=0$}\}$. 
Therefore
\begin{equation}\label{almostall} 
\{\alpha\in\DPS_n\mid \mbox{$0\in\dom(\alpha)$ and $0\alpha=0$}\}\cup \I(\{1,2,\ldots,n-1\}) 
\subseteq \langle\alpha_1,\alpha_2,\beta_1,\beta_2\rangle
\end{equation}
(since this union is clearly a submonoid of $\DPS_n$, which admits $\I(\{1,2,\ldots,n-1\})$ as an ideal, 
in fact the previous inclusion is an equality). 

\smallskip 

Now, let $i,j\in\{1,\ldots,n-1\}$. Then 
$$
\begin{pmatrix}0&i\\j&0\end{pmatrix}=\begin{pmatrix}0&i\\0&1\end{pmatrix}
\begin{pmatrix}0&1\\1&0\end{pmatrix}\begin{pmatrix}0&1\\0&j\end{pmatrix}, \quad 
\begin{pmatrix}0\\j\end{pmatrix}=\begin{pmatrix}0\\0\end{pmatrix}
\begin{pmatrix}0&1\\1&0\end{pmatrix}\begin{pmatrix}1\\j\end{pmatrix} \quad\text{and}\quad
\begin{pmatrix}i\\0\end{pmatrix}=\begin{pmatrix}i\\1\end{pmatrix}
\begin{pmatrix}0&1\\1&0\end{pmatrix}\begin{pmatrix}0\\0\end{pmatrix} 
$$
and so, also in view of (\ref{almostall}), we have  
$\begin{pmatrix}0&i\\j&0\end{pmatrix}, \begin{pmatrix}0\\j\end{pmatrix}, \begin{pmatrix}i\\0\end{pmatrix}\in\langle\alpha_1,\alpha_2,\beta_1,\beta_2,\gamma\rangle$ and thus  
we may conclude that 
$\DPS_n=\langle\alpha_1,\alpha_2,\beta_1,\beta_2,\gamma\rangle$, as required. 

\smallskip

Finally, regarding the case $n=3$, it suffices to notice that $\alpha_1=\alpha_2$ and $\beta_1=\gamma^2$. 
\end{proof}

Let $n\geqslant 3$ and take a set of generators $X$ of $\DPS_n$. 
Recall that, by Proposition \ref{subiso}, the submonoid 
$M=\{\alpha\in\DPS_n\mid \mbox{$0\in\dom(\alpha)$ and $0\alpha=0$}\}$ of $\DPS_n$ 
is isomorphic to $\I(\{1,2,\ldots,n-1\})$. 

For $n=3$, it is clear that 
$\alpha_1=\begin{pmatrix}0&1&2\\0&2&1\end{pmatrix}\in X$, 
since it is the only element of the group of units of $\DPS_3$ distinct from the identity. 

On the other hand, for $n\geqslant 4$, it is easy to check that an element of $M$
with rank greater than or equal to three can only be a product of elements belonging to $M$. 
Hence, in this case, $X$ must contain at least three elements of $M$ 
(at least two of them with rank $n$ and one of them with rank $n-1$). 

Next, observe that, for instance, $\gamma$ can only be obtained from $X$ if at least one element of the form 
$\begin{pmatrix}0&i\\j&0\end{pmatrix}$ belongs to $X$, for some $i,j\in\{1,\ldots,n-1\}$. 

Finally, since all elements of $M$ fix $0$, without at least one element of $\I(\{1,2,\ldots,n-1\})$ with rank $n-1$ in $X$, 
we cannot get, for instance, $\beta_2$ as a product of elements of $X$. 

Thus, we have proved the following result with which we end this section. 

\begin{theorem}
The rank of $\DPS_n$ is $3$, for $n=3$, and $5$, for $n\geqslant 4$.
\end{theorem}

\section{A Presentation for $\DPS_n$}
\label{dpspres}

We begin this section by recalling some notions related to the concept of a monoid presentation.

\smallskip 

Let $A$ be an alphabet and consider the free monoid $A^*$ generated by $A$. 
The elements of $A$ and of $A^*$ are called \textit{letters} and \textit{words}, respectively. 
The empty word is denoted by $1$ and we write $A^+$ to express $A^*\setminus\{1\}$. 
A pair $(u,v)$ of $A^*\times A^*$ is called a
\textit{relation} of $A^*$ and it is usually represented by $u=v$. 
To avoid confusion, given $u, v\in A^*$, we will write $u\equiv v$, instead
of  $u=v$, whenever we want to state precisely that $u$ and $v$
are identical words of $A^*$.  
A relation $u=v$ of $A^*$ is said to be a \textit{consequence} of $R$ if $u\mathrel{\rho_R} v$. 
Let $X$ be a generating set of $M$ and let $f: A\longrightarrow M$ be an injective mapping 
such that $Af=X$. 
Let $\varphi: A^*\longrightarrow M$ be the (surjective) homomorphism of monoids that extends $f$ to $A^*$. 
We say that $X$ satisfies (via $\varphi$) a relation $u=v$ of $A^*$ if $u\varphi=v\varphi$. 
For more details see
\cite{Lallement:1979} or \cite{Ruskuc:1995}. 
A direct method to find a presentation for a monoid
is described by the following well-known result (e.g.  see \cite[Proposition 1.2.3]{Ruskuc:1995}).  

\begin{proposition}\label{provingpresentation} 
Let $M$ be a monoid generated by a set $X$, let $A$ be an alphabet 
and let $f: A\longrightarrow M$ be an injective mapping 
such that $Af=X$. 
Let $\varphi:A^*\longrightarrow M$ be the (surjective) homomorphism 
that extends $f$ to $A^*$ and let $R\subseteq A^*\times A^*$.
Then $\langle A\mid R\rangle$ is a presentation for $M$ if and only
if the following two conditions are satisfied:
\begin{enumerate}
\item
The generating set $X$ of $M$ satisfies (via $\varphi$) all the relations from $R$;  
\item 
If $u,v\in A^*$ are any two words such that 
the generating set $X$ of $M$ satisfies (via $\varphi$) the relation $u=v$ then $u=v$ is a consequence of $R$.
\end{enumerate} 
\end{proposition}

\smallskip

Given a presentation for a monoid, another method to find a new
presentation consists in applying Tietze transformations. For a
monoid presentation $\langle A\mid R\rangle$, the 
four \emph{elementary Tietze transformations} are:

\begin{description}
\item(T1)
Adding a new relation $u=v$ to $\langle A\mid R\rangle$,
provided that $u=v$ is a consequence of $R$;
\item(T2)
Deleting a relation $u=v$ from $\langle A\mid R\rangle$,
provided that $u=v$ is a consequence of $R\backslash\{u=v\}$;
\item(T3)
Adding a new generating symbol $b$ and a new relation $b=w$, where
$w\in A^*$;
\item(T4)
If $\langle A\mid R\rangle$ possesses a relation of the form
$b=w$, where $b\in A$, and $w\in(A\backslash\{b\})^*$, then
deleting $b$ from the list of generating symbols, deleting the
relation $b=w$, and replacing all remaining appearances of $b$ by
$w$.
\end{description}

The next result is well-known (e.g. see \cite{Ruskuc:1995}): 

\begin{proposition} \label{tietze}
Two finite presentations define the same monoid if and only if one
can be obtained from the other by a finite number of elementary
Tietze transformations $(T1)$, $(T2)$, $(T3)$ and $(T4)$.  
\end{proposition}

\medskip 

In this section, we aim to determine a presentation for $\DPS_n$. 
In order to achieve this objective, we will take into account known presentations of symmetric inverse monoids.
So, we begin by recalling the following well known presentation of the symmetric inverse monoid $\I(\{1,2,\ldots,n-1\})$, for $n\geqslant 4$: 
\begin{multline}\label{presin} 
\langle a_1,a_2,b \mid  
a_2^2=a_1^{n-1}=(a_1a_2)^{n-2}=(a_2a_1^{n-2}a_2a_1)^3=1,\: (a_2a_1^{n-1-j}a_2a_1^j)^2=1 \:(2\leqslant j\leqslant n-3),
\\
a_1^{n-2}a_2a_1ba_1^{n-2}a_2a_1=a_1a_2ba_2a_1^{n-2}=b=b^2, \: (ba_2)^2=ba_2b=(a_2b)^2\rangle.
\end{multline}
This presentation is associated to the set of generators 
$$
\left\{\alpha'_1=\begin{pmatrix}1&2&\cdots&n-2&n-1\\2&3&\cdots&n-1&1\end{pmatrix}, \: 
\alpha'_2=\begin{pmatrix}1&2&3&\cdots&n-1\\2&1&3&\cdots&n-1\end{pmatrix}, \: 
\beta'=\begin{pmatrix}2&3&\cdots&n-1\\2&3&\cdots&n-1\end{pmatrix}\right\}
$$ 
of $\I(\{1,2,\ldots,n-1\})$ via the homomorphism of monoids $\{a_1,a_2,b\}^*\longrightarrow \I(\{1,2,\ldots,n-1\})$ that extends the mapping  
$a_1\longmapsto\alpha'_1$, $a_2\longmapsto\alpha'_2$ and $b\longmapsto\beta'$ (see \cite{Fernandes:2002survey}). 

\smallskip 

Next, by applying Tietze transformations, we deduce a presentation of $\I(\{1,2,\ldots,n-1\})$ 
associated to the following set of generators:  
$$
\left\{\alpha'_1=\begin{pmatrix}1&2&\cdots&n-2&n-1\\2&3&\cdots&n-1&1\end{pmatrix}, \: 
\alpha'_2=\begin{pmatrix}1&2&3&\cdots&n-1\\2&1&3&\cdots&n-1\end{pmatrix}, \: 
\beta'_1=\begin{pmatrix}1&2&\cdots&n-2\\1&2&\cdots&n-2\end{pmatrix}\right\}. 
$$ 
Notice that $\beta'_1=\alpha'_1\beta'{\alpha'_1}^{n-2}$ and $\beta'={\alpha'_1}^{n-2}\beta'_1\alpha'_1$. 

\begin{proposition}\label{presin-1}
For $n\geqslant 4$, the monoid $\I(\{1,2,\ldots,n-1\})$ is defined by the presentation
\begin{multline}\nonumber
\langle a_1,a_2,b_1\mid a_2^2=a_1^{n-1}=(a_1a_2)^{n-2}=(a_2a_1^{n-2}a_2a_1)^3=1,\:(a_2a_1^{n-1-j}a_2a_1^j)^2=1 \:(2\leqslant j\leqslant n-3),
\\
a_1a_2a_1^{n-2}b_1a_1a_2a_1^{n-2}=a_1^{n-2}b_1a_1,\: (a_1^{n-2}b_1a_1a_2)^2=a_1^{n-2}b_1a_1a_2a_1^{n-2}b_1a_1=(a_2a_1^{n-2}b_1a_1)^2,
\\
b_1^2=b_1, \: a_2b_1=b_1a_2\rangle,  
\end{multline}
which is associated to its set of generators $\{\alpha'_1,\alpha'_2,\beta'_1\}$ 
via the homomorphism of monoids $\{a_1,a_2,b_1\}^*\longrightarrow \I(\{1,2,\ldots,n-1\})$ that extends the mapping  
$a_1\longmapsto\alpha'_1$, $a_2\longmapsto\alpha'_2$ and $b_1\longmapsto\beta'_1$. 
\end{proposition} 

\begin{proof}
We proceed by applying elementary Tietze transformations to the above presentation (\ref{presin}). 

\smallskip

Step 1: We add a new symbol, $b_1$, to the alphabet and add the new relation $b_1=a_1ba_1^{n-2}$. The resulting presentation is 
\begin{multline}\nonumber
\langle a_1,a_2,b,b_1\mid a_2^2=a_1^{n-1}=(a_1a_2)^{n-2}=(a_2a_1^{n-2}a_2a_1)^3=1,\: (a_2a_1^{n-1-j}a_2a_1^j)^2=1 \:(2\leqslant j\leqslant n-3),
\\
a_1^{n-2}a_2a_1ba_1^{n-2}a_2a_1=a_1a_2ba_2a_1^{n-2}=b=b^2,\: (ba_2)^2=ba_2b=(a_2b)^2, \: b_1=a_1ba_1^{n-2}\rangle. 
\end{multline}

\smallskip

Step 2: We add a new relation $b=a_1^{n-2}b_1a_1$. Observe that $b=a_1^{n-2}b_1a_1$ is a consequence of 
the relations $a_1^{n-1}=1$ and $b_1=a_1ba_1^{n-2}$: 
$$
b=1b1=a_1^{n-2}a_1ba_1^{n-2}a_1=a_1^{n-2}b_1a_1.
$$
The resulting presentation is 
\begin{multline}\nonumber
\langle a_1,a_2,b,b_1\mid a_2^2=a_1^{n-1}=(a_1a_2)^{n-2}=(a_2a_1^{n-2}a_2a_1)^3=1,\: (a_2a_1^{n-1-j}a_2a_1^j)^2=1 \:(2\leqslant j\leqslant n-3),
\\
a_1^{n-2}a_2a_1ba_1^{n-2}a_2a_1=a_1a_2ba_2a_1^{n-2}=b=b^2,\: (ba_2)^2=ba_2b=(a_2b)^2, \: 
b_1=a_1ba_1^{n-2}, \: b=a_1^{n-2}b_1a_1\rangle. 
\end{multline}

\smallskip 

Step 3: We remove the symbol $b$, along with the relation $b=a_1^{n-2}b_1a_1$, and replace all occurrences of $b$ by $a_1^{n-2}b_1a_1$ in the remaining relations. The resulting presentation is 
\begin{multline}\nonumber
\langle a_1,a_2,b_1\mid a_2^2=a_1^{n-1}=(a_1a_2)^{n-2}=(a_2a_1^{n-2}a_2a_1)^3=1,\: (a_2a_1^{n-1-j}a_2a_1^j)^2=1 \:(2\leqslant j\leqslant n-3),
\\
a_1^{n-2}a_2a_1^{n-1}b_1a_1^{n-1}a_2a_1=a_1a_2a_1^{n-2}b_1a_1a_2a_1^{n-2}=a_1^{n-2}b_1a_1=(a_1^{n-2}b_1a_1)^2,
\\
(a_1^{n-2}b_1a_1a_2)^2=a_1^{n-2}b_1a_1a_2a_1^{n-2}b_1a_1=(a_2a_1^{n-2}b_1a_1)^2, \: b_1=a_1^{n-1}b_1a_1^{n-1}\rangle;
\end{multline}

\smallskip

Step 4: We add the relations $b_1^2=b_1$ and $a_2b_1=b_1a_2$, as a result of being consequences of 
$a_1^{n-1}=1$, $a_2^2=1$, $a_1^{n-2}a_2a_1^{n-1}b_1a_1^{n-1}a_2a_1=a_1^{n-2}b_1a_1$ and $a_1^{n-2}b_1a_1=(a_1^{n-2}b_1a_1)^2$: 
$$
b_1^2=1b_11b_11=a_1^{n-1}b_1a_1^{n-1}b_1a_1^{n-1}=a_1(a_1^{n-2}b_1a_1)^2a_1^{n-2}=a_1a_1^{n-2}b_1a_1a_1^{n-2}=1b_11=b_1
$$
and
\begin{multline}\nonumber
a_2b_1=1a_21b_11=a_1^{n-1}a_2a_1^{n-1}b_1a_1^{n-1}1=a_1^{n-1}a_2a_1^{n-1}b_1a_1^{n-1}a_2^2=a_1^{n-1}a_2a_1^{n-1}b_1a_1^{n-1}a_21a_2
\\
= a_1^{n-1}a_2a_1^{n-1}b_1a_1^{n-1}a_2a_1^{n-1}a_2=a_1a_1^{n-2}b_1a_1a_1^{n-2}a_2=1b_11a_2=b_1a_2.
\end{multline}
The resulting presentation is 
\begin{multline}\nonumber
\langle a_1,a_2,b_1\mid a_2^2=a_1^{n-1}=(a_1a_2)^{n-2}=(a_2a_1^{n-2}a_2a_1)^3=1, \: (a_2a_1^{n-1-j}a_2a_1^j)^2=1 \:(2\leqslant j\leqslant n-3),
\\
a_1^{n-2}a_2a_1a_1^{n-2}b_1a_1a_1^{n-2}a_2a_1=a_1a_2a_1^{n-2}b_1a_1a_2a_1^{n-2}=a_1^{n-2}b_1a_1=(a_1^{n-2}b_1a_1)^2,
\\
(a_1^{n-2}b_1a_1a_2)^2=a_1^{n-2}b_1a_1a_2a_1^{n-2}b_1a_1=(a_2a_1^{n-2}b_1a_1)^2, \: b_1=a_1^{n-1}b_1a_1^{n-1}, \: 
 b_1^2=b_1, \: a_2b_1=b_1a_2\rangle.
\end{multline}

\smallskip

Step 5: We may remove the relations 
$a_1^{n-2}a_2a_1a_1^{n-2}b_1a_1a_1^{n-2}a_2a_1=a_1^{n-2}b_1a_1$, $a_1^{n-2}b_1a_1=(a_1^{n-2}b_1a_1)^2$ and $b_1=a_1^{n-1}b_1a_1^{n-1}$, since they are consequences of $a_1^{n-1}=1$, $a_2^2=1$, $b_1^2=b_1$ and $a_2b_1=b_1a_2$:
$$
a_1^{n-2}a_2a_1a_1^{n-2}b_1a_1a_1^{n-2}a_2a_1=a_1^{n-2}a_21b_11a_2a_1=a_1^{n-2}a_2b_1a_2a_1=a_1^{n-2}b_1a_2^2a_1
= a_1^{n-2}b_11a_1=a_1^{n-2}b_1a_1, 
$$
$$
a_1^{n-2}b_1a_1=a_1^{n-2}b_1^2a_1=a_1^{n-2}b_11b_1a_1=a_1^{n-2}b_1a_1^{n-1}b_1a_1=(a_1^{n-2}b_1a_1)^2
$$
and
$$
b_1=1b_11=a_1^{n-1}b_1a_1^{n-1}.
$$
The resulting presentation is 
\begin{multline}\nonumber
\langle a_1,a_2,b_1\mid a_2^2=a_1^{n-1}=(a_1a_2)^{n-2}=(a_2a_1^{n-2}a_2a_1)^3=1,\: (a_2a_1^{n-1-j}a_2a_1^j)^2=1 \:(2\leqslant j\leqslant n-3),
\\
a_1a_2a_1^{n-2}b_1a_1a_2a_1^{n-2}=a_1^{n-2}b_1a_1, \: (a_1^{n-2}b_1a_1a_2)^2=a_1^{n-2}b_1a_1a_2a_1^{n-2}b_1a_1=(a_2a_1^{n-2}b_1a_1)^2, 
\\  
b_1^2=b_1, \: a_2b_1=b_1a_2\rangle, 
\end{multline}
as required. 
\end{proof}

Now, recall that, by Proposition \ref{subiso}, for $n\geqslant 3$, 
$\{\alpha\in\DPS_n\mid \mbox{$0\in\dom(\alpha)$ and $0\alpha=0$}\}$ is a submonoid of $\DPS_n$ 
isomorphic to the symmetric inverse monoid 
$\I(\{1,2,\ldots,n-1\})$: 
$$
\I(\{1,2,\ldots,n-1\})\psi=\{\alpha\in\DPS_n\mid \mbox{$0\in\dom(\alpha)$ and $0\alpha=0$}\}, 
$$
where $\psi: \I(\{1,2,\ldots,n-1\}) \longrightarrow\I(\{0,1,2,\ldots,n-1\})$ is the injective homomorphism of monoids defined in the proof of 
Proposition \ref{subiso}. 
Since $\alpha'_1\psi=\alpha_1$, $\alpha'_2\psi=\alpha_2$ and $\beta'_1\psi=\beta_1$,  
as an immediate consequence of Proposition \ref{presin-1}, we have: 

\begin{corollary}\label{presisoin-1}
For $n\geqslant 4$, the submonoid $\{\alpha\in\DPS_n\mid \mbox{$0\in\dom(\alpha)$ and $0\alpha=0$}\}$ of $\DPS_n$ is defined by the presentation
\begin{multline}\nonumber
\langle a_1,a_2,b_1\mid a_2^2=a_1^{n-1}=(a_1a_2)^{n-2}=(a_2a_1^{n-2}a_2a_1)^3=1,\:(a_2a_1^{n-1-j}a_2a_1^j)^2=1 \:(2\leqslant j\leqslant n-3),
\\
a_1a_2a_1^{n-2}b_1a_1a_2a_1^{n-2}=a_1^{n-2}b_1a_1,\: (a_1^{n-2}b_1a_1a_2)^2=a_1^{n-2}b_1a_1a_2a_1^{n-2}b_1a_1=(a_2a_1^{n-2}b_1a_1)^2,
\\
b_1^2=b_1, \: a_2b_1=b_1a_2\rangle,   
\end{multline}
which is associated to its set of generators $\{\alpha_1,\alpha_2,\beta_1\}$ 
via the homomorphism of monoids $\{a_1,a_2,b_1\}^*\longrightarrow \I(\{1,2,\ldots,n-1\})$ that extends the mapping  
$a_1\longmapsto\alpha_1$, $a_2\longmapsto\alpha_2$ and $b_1\longmapsto\beta_1$. 
\end{corollary} 

\smallskip 

Next, let $n\geqslant 4$ and 
consider the alphabet $A=\{a_1,a_2,b_1,b_2,c\}$ and the set $R$ formed by the following $3n+9$ monoid relations:
\begin{enumerate}
\item[$(R_1)$] $a_2^2=1$;
\item[$(R_2)$] $a_1^{n-1}=1$;
\item[$(R_3)$] $(a_1a_2)^{n-2}=1$;
\item[$(R_4)$] $(a_2a_1^{n-2}a_2a_1)^3=1$;
\item[$(R_5)$] $(a_2a_1^{n-1-j}a_2a_1^j)^2=1$, $j=2,\ldots,n-3$;
\item[$(R_6)$] $b_1^2=b_1$ and $b_2^2=b_2$;
\item[$(R_7)$] $a_2b_1=b_1a_2$, $b_2a_2=a_2b_2$, $b_2a_1=a_1b_2$ and $b_2b_1=b_1b_2$;
\item[$(R_{8})$] $a_1a_2a_1^{n-2}b_1a_1a_2a_1^{n-2}=a_1^{n-2}b_1a_1$;
\item[$(R_{9})$] $(a_1^{n-2}b_1a_1a_2)^2=(a_2a_1^{n-2}b_1a_1)^2$;
\item[$(R_{10})$] $a_1^{n-2}b_1a_1a_2a_1^{n-2}b_1a_1=(a_2a_1^{n-2}b_1a_1)^2$;
\item[$(R_{11})$] $c^3=c$;
\item[$(R_{12})$] $ca_1=ca_2$;
\item[$(R_{13})$] $a_1^{n-2}c=a_2c$;
\item[$(R_{14})$] $a_2a_1^jc=a_1^jc$, $j=1,\ldots,n-3$;
\item[$(R_{15})$] $b_1a_1c=cb_2$;
\item[$(R_{16})$] $b_1a_1^jc=a_1^jc$, $j=2,\ldots,n-3$;
\item[$(R_{17})$] $(b_1a_1)^{n-3}b_1=c^2a_2a_1^{n-4}$;
\item[$(R_{18})$] $b_2c^2=ca_2c$;
\item[$(R_{19})$] $(b_2c)^2=b_2cb_2$.
\end{enumerate}

Our goal now is to show that the monoid $\DPS_n$ is defined by the presentation $\langle A\mid R\rangle$.

\smallskip 

Let $f:A\longrightarrow \DPS_n$ be the mapping defined by
$$
a_1f=\alpha_1 ,\quad a_2f=\alpha_2 ,\quad b_1f=\beta_1 ,\quad b_2f=\beta_2 \quad\text{and}\quad cf=\gamma
$$
and let $\varphi:A^*\longrightarrow \DPS_n$ be the homomorphism of monoids that extends $f$ to $A^*$.

\smallskip 

First of all, it is a routine matter to check that:

\begin{lemma}\label{genrel}
The set of generators $\{\alpha_1,\alpha_2,\beta_1,\beta_2,\gamma\}$ of $\DPS_n$ satisfies (via $\varphi$) all the relations from $R$.
\end{lemma}

Notice that the previous lemma assures us that, if $w_1,w_2\in\{a_1,a_2,b_1,b_2,c\}^*$ are such that $w_1=w_2$ is a consequence of $R$, then $w_1\varphi=w_2\varphi$.

\smallskip 

The following lemma is an immediate consequence 
of Proposition \ref{provingpresentation} and Corollary \ref{presisoin-1}.

\begin{lemma}\label{a1a2b1}
Let $w_1, w_2\in\{a_1,a_2,b_1\}^*$. If $w_1\varphi=w_2\varphi$ then $w_1=w_2$ is a consequence of $R$.
\end{lemma}

Our next lemma provides us some useful relations that are consequence of $R$. 

\begin{lemma}\label{rconseq}
One has: 
\begin{enumerate}
\item The relation $a_1a_2c=c$ is a consequence of $R$;
\item The relation $c^2=(b_1a_1)^{n-3}b_1a_1^3a_2$ is a consequence of $R$;
\item The relation $b_1c=c$ is a consequence of $R$;
\item The relation $b_2c=ca_2(b_1a_1)^{n-3}b_1a_1^3a_2$ is a consequence of $R$.
\end{enumerate}
\end{lemma}

\begin{proof}
1. It follows from relations $(R_2)$ and $(R_{13})$ that
$a_1a_2c=a_1a_1^{n-2}c=1c=c$, 
which implies that $a_1a_2c=c$ is a consequence of $R$.

2. From relations $(R_1)$, $(R_2)$ and $(R_{17})$ we can deduce that $c^2=(b_1a_1)^{n-3}b_1a_1^3a_2$ is a consequence of $R$, since 
$$
c^2=c^21=c^2a_2^2=c^2a_21a_2=c^2a_2a_1^{n-1}a_2=(b_1a_1)^{n-3}b_1a_1^3a_2.
$$ 

3. If we consider the relations $(R_6)$, $(R_{11})$ and (from 2) $c^2=(b_1a_1)^{n-3}b_1a_1^3a_2$, then we obtain
$$
b_1c=b_1c^3=b_1(b_1a_1)^{n-3}b_1a_1^3a_2c=(b_1a_1)^{n-3}b_1a_1^3a_2c=c^3=c, 
$$ 
whence $b_1c=c$ is a consequence of $R$.

4. Finally, by considering the relations $(R_{11})$, $(R_{18})$ and $c^2=(b_1a_1)^{n-3}b_1a_1^3a_2$, we get
$$
b_2c=b_2c^3=ca_2c^2=ca_2(b_1a_1)^{n-3}b_1a_1^3a_2
$$ 
and so $b_2c=ca_2(b_1a_1)^{n-3}b_1a_1^3a_2$ is a consequence of $R$, as required. 
\end{proof}

Let $w\in\{a_1,a_2,b_1,b_2,c\}^*$ and $x\in\{a_1,a_2,b_1,b_2,c\}$. We denote by $|w|_x$ the number of occurrences of the letter $x$ in the word $w$.

\begin{lemma}\label{0b2}
Let $w\in\{a_1,a_2,b_1,b_2\}^*$. Then $|w|_{b_2}=0$ if and only if $0\in\dom(w\varphi)$.
\end{lemma}

\begin{proof}
First, suppose $|w|_{b_2}=0$. Then $w\in\{a_1,a_2,b_1\}^*$, which implies that 
$w\varphi\in\langle\alpha_1,\alpha_2,\beta_1\rangle = \{\alpha\in\DPS_n\mid \mbox{$0\in\dom(\alpha)$ and $0\alpha=0$}\}$, 
whence $0\in\dom(w\varphi)$.

Conversely, admit that $|w|_{b_2}\geqslant 1$. It follows from relations $(R_6)$ and $(R_7)$ that $w=b_2w$ is a consequence of $R$. Then $w\varphi=(b_2w)\varphi=(b_2\varphi)(w\varphi)=\beta_2(w\varphi)$. Therefore, $\dom(w\varphi)\subseteq\dom(\beta_2)=\{1,2, \ldots,n-1\}$ and so $0\notin\dom(w\varphi)$.
\end{proof}

\begin{lemma}\label{b2a1a2b1}
Let $w_1,w_2\in\{a_1,a_2,b_1\}^*$. If $(b_2w_1)\varphi=(b_2w_2)\varphi$ then $w_1\varphi=w_2\varphi$.
\end{lemma}

\begin{proof}
It suffices to observe that 
$((b_2w_i)\varphi)\psi=(\beta_2(w_i\varphi))\psi=(w_i\varphi|_{\dom(w_i\varphi)\setminus\{0\}})\psi = w_i\varphi$, for $i=1,2$. 
\end{proof}

Now, we can prove: 

\begin{lemma}\label{a1a2b1b2}
Let $w_1, w_2\in\{a_1,a_2,b_1,b_2\}^*$. If $w_1\varphi=w_2\varphi$ then $w_1=w_2$ is a consequence of $R$.
\end{lemma}

\begin{proof}
First, observe that, by Lemma \ref{0b2}, we have 
$|w_1|_{b_2}=0$ if and only if $|w_2|_{b_2}=0$. 

If $|w_1|_{b_2}=|w_2|_{b_2}=0$ then $w_1,w_2\in\{a_1,a_2,b_1\}^*$ and so, by Lemma \ref{a1a2b1}, 
$w_1=w_2$ is a consequence of $R$. 

On the other hand, admit that $|w_1|_{b_2}, |w_2|_{b_2}>0$. 
Then, it follows from relations $(R_6)$ and $(R_7)$ that there exist $w'_1,w'_2\in\{a_1,a_2,b_1\}^*$ such that $w_1=b_2w'_1$ and $w_2=b_2w'_2$ are consequences of $R$. Hence $(b_2w'_1)\varphi=w_1\varphi=w_2\varphi=(b_2w'_2)\varphi$. 
Thus, by Lemma \ref{b2a1a2b1}, 
we have $w'_1\varphi=w'_2\varphi$ and so, by Lemma \ref{a1a2b1}, we conclude that $w'_1=w'_2$ is a consequence of $R$. Therefore, $b_2w'_1=b_2w'_2$ is a consequence of $R$, which implies that $w_1=w_2$ is a consequence of $R$, as required.
\end{proof}

Next, we continue with a series of lemmas now also involving the letter $c$.

\begin{lemma} \label{cwc=w1cw2c}
Let $w\in\{a_1,a_2,b_1,b_2\}^*$. Then there exist $w_1,w_2\in\{a_1,a_2,b_1\}^*$ such that $cwc=w_1cw_2c$ is a consequence of $R$.
\end{lemma}

\begin{proof}
We divide this proof into three cases.

\smallskip 

Case 1: If $w\in\{a_1,a_2,b_1\}^*$ then it suffices to take $w_1\equiv 1$ and $w_2 \equiv w$.

\smallskip

Case 2: Assume $w\in\{b_2\}^+$. 
It follows from relations $(R_6)$ that $cwc=cb_2c$ is a consequence of $R$. 
Then, if we take $w_1\equiv b_1a_1$ and $w_2\equiv 1$ and consider the relation $(R_{15})$, 
we obtain
$$
cwc=cb_2c=cb_21c=b_1a_1c1c=w_1cw_2c
$$
and so $cwc=w_1cw_2$ is a consequence of $R$.

\smallskip

Case 3: Assume $w\in\{a_1,a_2,b_1,b_2\}^+\setminus(\{b_2\}^+\cup\{a_1,a_2,b_1\}^+)$. 
It follows from relations $(R_6)$ and $(R_7)$ that there exists $w_2 \in\{a_1,a_2,b_1\}^+$ such that $cwc=cb_2w_2c$ is a consequence of $R$.
Then, being $w_1\equiv b_1a_1$, by applying the relation $(R_{15})$, we have 
$$
cwc=cb_2w_2c=b_1a_1cw_2c=w_1cw_2c,
$$
which implies that $cwc=w_1cw_2c$ is a consequence of $R$.
\end{proof}

\begin{lemma}\label{cwc=w1c2w2}
Let $w\in\{a_1,a_2,b_1,b_2\}^*$. Then there exist $w_1,w_2\in\{a_1,a_2,b_1,b_2\}^*$ such that $cwc=w_1c^2w_2$ is a consequence of $R$.
\end{lemma}
\begin{proof}
By Lemma \ref{cwc=w1cw2c} there exist $u_1,u_2\in\{a_1,a_2,b_1\}^*$ such that $cwc=u_1cu_2c$ is a consequence of $R$. 
We complete the proof by showing that there exist $w_1,w_2\in\{a_1,a_2,b_1,b_2\}^*$ such that $u_1cu_2c=w_1c^2w_2$ is a consequence of $R$. 
We will proceed by induction on the length $|u_2|$ of $u_2$.

Suppose that $|u_2|=0$. Then $u_2\equiv 1$. It is clear that, if we take $w_1\equiv u_1$ and $w_2\equiv 1$, then $u_1cu_2c=w_1c^2w_2$ is a consequence of $R$.

Let $k\geqslant 1$ and assume that, for all $u\in\{a_1,a_2,b_1\}^*$ such that $|u|<k$, there exist $w_1,w_2\in\{a_1,a_2,b_1,b_2\}^*$ such that $u_1cuc=w_1c^2w_2$.

Suppose that $|u_2|=k$. 
As a consequence of relations $(R_1)$, $(R_2)$ and $(R_6)$, we deduce that $u_2=u_3$ is a consequence of $R$, 
for some subword $u_3\in\{a_1,a_2,b_1\}^*$ of $u_2$ such that none of the words of 
$\{a_1^i \mid i\geqslant n-1\}\cup\{a_2^i \mid i\geqslant 2\}\cup\{b_1^i \mid i\geqslant 2\}$ is a factor of $u_3$. 
Notice that $|u_3|\leqslant |u_2|$.

If $|u_3|<|u_2|$ then, by the induction hypothesis, there exist $w_1,w_2\in\{a_1,a_2,b_1,b_2\}^*$ such that $u_1cu_3c=w_1c^2w_2$ is a consequence of $R$. Clearly, since $u_2=u_3$ is a consequence of $R$, we have that $u_1cu_2c=u_1cu_3c$ is a consequence of $R$ and so $u_1cu_2c=w_1c^2w_2$ is a consequence of $R$.

Now, consider that $|u_3|=|u_2|$. Thus $u_2\equiv u_3$ and, consequently, none of the words of 
$\{a_1^i \mid i\geqslant n-1\}\cup\{a_2^i \mid i\geqslant 2\}\cup\{b_1^i \mid i\geqslant 2\}$ is a factor of $u_2$.

Let $u_4\in\{a_1,a_2,b_1\}^+$ be a suffix of $u_2$ and let $u_5\in\{a_1,a_2,b_1\}^*$ be such that $u_2=u_5u_4$. Then, we can choose $u_4$ and $u_5$  satisfying one of the following cases, which we study separately, concluding the proof.

\smallskip

Case 1: Assume that $u_4\equiv b_1$. By Lemma \ref{rconseq}, $b_1c=c$ is a consequence of $R$ and so we have 
$$
u_1cu_2c=u_1cu_5u_4c=u_1cu_5b_1c=u_1cu_5c,
$$
whence $u_1cu_2c=u_1cu_5c$ is a consequence of $R$. 
Since $|u_5|=|u_2|-1<|u_2|$, by the induction hypothesis there exist $w_1,w_2\in\{a_1,a_2,b_1,b_2\}^*$ such that $u_1cu_5c=w_1c^2w_2$ is a consequence of $R$.
Thus $u_1cu_2c=w_1c^2w_2$ is a consequence of $R$.

\smallskip

Case 2: Suppose that $u_4\equiv a_2$ and $u_5\equiv 1$. Then $u_2\equiv a_2$ and, by relation $(R_{18})$, we have 
$$
u_1cu_2c=u_1ca_2c=u_1b_2c^2=w_1c^2w_2,
$$
where $w_1 \equiv u_1b_2$ and $w_2 \equiv 1$. Hence $u_1cu_2c=w_1c^2w_2$ is a consequence of $R$.

\smallskip

Case 3: Take $u_4\equiv a_1a_2$. It follows from Lemma \ref{rconseq} that $a_1a_2c=c$ is a consequence of $R$ and so 
$$
u_1cu_2c=u_1cu_5u_4c=u_1cu_5a_1a_2c=u_1cu_5c.
$$
Then, $u_1cu_2c=u_1cu_5c$ is a consequence of $R$. 
Observe that $|u_5|=|u_2|-2<|u_2|$. Hence, by the induction hypothesis, 
there exist $w_1,w_2\in\{a_1,a_2,b_1,b_2\}^*$ such that $u_1cu_5c=w_1c^2w_2$ is a consequence of $R$.
Therefore, $u_1cu_2c=w_1c^2w_2$ is a consequence of $R$.

\smallskip

Case 4: Let $u_4\equiv b_1a_2$. By Lemma \ref{rconseq} we have that $b_1c=c$ is a consequence of $R$. 
By considering the relations $(R_7)$, we obtain 
$$
u_1cu_2c=u_1cu_5u_4c=u_1cu_5b_1a_2c=u_1cu_5a_2b_1c=u_1cu_5a_2c.
$$
Then, $u_1cu_2c=u_1cu_5a_2c$ is a consequence of $R$. 
Notice that $|u_5a_2|=|u_2|-1<|u_2|$. So, by the induction hypothesis,  
there exist $w_1,w_2\in\{a_1,a_2,b_1,b_2\}^*$ such that $u_1cu_5a_2c=w_1c^2w_2$ is a consequence of $R$.
Hence, $u_1cu_2c=w_1c^2w_2$ is a consequence of $R$.

\smallskip

Case 5: Assume that  $u_4\equiv a_1$ and $u_5\equiv 1$. Then $u_2\equiv a_1$. From relation $(R_{12})$ we have 
$$
u_1cu_2c=u_1ca_1c=u_1ca_2c,
$$
which implies that $u_1cu_2c=u_1ca_2c$ is a consequence of $R$. By Case 2 there exist $w_1,w_2\in\{a_1,a_2,b_1,b_2\}^*$ such that $u_1ca_2c=w_1c^2w_2$ is a consequence of $R$. Hence $u_1cu_2c=w_1c^2w_2$ is a consequence of $R$.

\smallskip

Case 6: Suppose that  $u_4\equiv a_1^j$, for some $j\in\{2,\ldots,n-3\}$, and $u_5\equiv 1$. Then $u_2\equiv a_1^j$ and it follows from relations $(R_{12})$ and $(R_{14})$ that
$$
u_1cu_2c=u_1ca_1^jc=u_1ca_2a_1^{j-1}c=u_1ca_1^{j-1}c.
$$
Thus, $u_1cu_2c=u_1ca_1^{j-1}c$ is a consequence of $R$. 
Since $|a_1^{j-1}|=|u_2|-1<|u_2|$, we can use the induction hypothesis to conclude that 
there exist $w_1,w_2\in\{a_1,a_2,b_1,b_2\}^*$ such that $u_1ca_1^{j-1}c=w_1c^2w_2$ is a consequence of $R$, 
which implies that $u_1cu_2c=w_1c^2w_2$ is a consequence of $R$.

\smallskip

Case 7: Assume that $u_4\equiv a_1^{n-2}$. Then, we have
$$
u_1cu_2c=u_1cu_5u_4c=u_1cu_5a_1^{n-2}c=u_1cu_5a_2c,
$$
by applying the relation $(R_{13})$. So, $u_1cu_2c=u_1cu_5a_2c$ is a consequence of $R$. 
As $|u_5a_2|=|u_2|-(n-2)+1=|u_2|-(n-3)<|u_2|$ then, by the induction hypothesis, there exist $w_1,w_2\in\{a_1,a_2,b_1,b_2\}^*$ such that $u_1cu_5a_2c=w_1c^2w_2$ is a consequence of $R$. Thus, $u_1cu_2c=w_1c^2w_2$ is a consequence of $R$.

\smallskip

Case 8: Take $u_4\equiv a_2a_1^j$, for some $j\in\{1,\ldots,n-3\}$. Then, by the relations $(R_{14})$,  we have
$$
u_1cu_2c=u_1cu_5u_4c=u_1cu_5a_2a_1^jc=u_1cu_5a_1^jc.
$$
Thus, $u_1cu_2c=u_1cu_5a_1^jc$ is a consequence of $R$. 
Since $|u_5a_1^j|=|u_2|-1<|u_2|$, the induction hypothesis assures us that there exist $w_1,w_2\in\{a_1,a_2,b_1,b_2\}^*$ such that $u_1cu_5a_1^jc=w_1c^2w_2$ is a consequence of $R$, which implies that $u_1cu_2c=w_1c^2w_2$ is a consequence of $R$.

\smallskip

Case 9: Let $u_4\equiv b_1a_1$. Then, we have
$$
u_1cu_2c=u_1cu_5u_4c=u_1cu_5b_1a_1c=u_1cu_5cb_2,
$$
which follows from relation $(R_{15})$. Hence, $u_1cu_2c=u_1cu_5cb_2$ is a consequence of $R$. 
Since $|u_5|=|u_2|-2<|u_2|$, by the induction hypothesis, there exist $w_1,w'_2\in\{a_1,a_2,b_1,b_2\}^*$ such that $u_1cu_5c=w_1c^2w'_2$ is a consequence of $R$ and so $u_1cu_2c=w_1c^2w_2$ is a consequence of $R$, where $w_2\equiv w'_2b_2$.

\smallskip

Case 10: Finally, assume that $u_4\equiv b_1a_1^j$, for some $j\in\{2,\ldots,n-3\}$. It follows from relations $(R_{16})$ that
$$
u_1cu_2c=u_1cu_5u_4c=u_1cu_5b_1a_1^jc=u_1cu_5a_1^jc,
$$
whence $u_1cu_2c=u_1cu_5a_1^jc$ is a consequence of $R$. 
As $|u_5a_1^j|=|u_2|-1<|u_2|$ then, by the induction hypothesis, there exist $w_1,w_2\in\{a_1,a_2,b_1,b_2\}^*$ 
such that $u_1cu_5a_1^jc=w_1c^2w_2$ is a consequence of $R$. 
Thus, $u_1cu_2c=w_1c^2w_2$ is a consequence of $R$, as required.
\end{proof}

\begin{lemma}\label{w=ww=w1cw2}
Let $w\in\{a_1,a_2,b_1,b_2,c\}^*\setminus\{a_1,a_2,b_1,b_2\}^*$.
\begin{enumerate}
\item If $|w|_c$ is even, then there exists $w'\in\{a_1,a_2,b_1,b_2\}^*$ such that $w=w'$ is a consequence of $R$;
\item If $|w|_c$ is odd, then there exist $w_1,w_2\in\{a_1,a_2,b_1,b_2\}^*$ such that $w=w_1cw_2$ is a consequence of $R$.
\end{enumerate}
\end{lemma}

\begin{proof}
First, we prove the lemma, by induction on $m=|w|_c$, 
for words  $w\in\{a_1,a_2,b_1,b_2,c\}^*\setminus\{a_1,a_2,b_1,b_2\}^*$ of the form 
$$
w\equiv v_0cv_1cv_2c\cdots cv_{m-1}cv_m,
$$ 
for some $m\in\mathbb{N}$, $v_0,v_m\in\{a_1,a_2,b_1,b_2\}^*$ and $v_1,v_2,\ldots,v_{m-1}\in\{a_1,a_2,b_1,b_2\}^+$.

If $m=1$ then $w\equiv v_0cv_1$ and the result follows trivially. 

Thus, let $m\geqslant 2$ and assume that the result is valid for $m-1$.

Take $v\equiv v_1cv_2c\cdots cv_{m-1}cv_m$. Then $w\equiv v_0cv$.

Admit that $|w|_c$ is even.  Then $|v|_c$ is odd. Hence, by the induction hypothesis, there exist $w_1,w_2\in\{a_1,a_2,b_1,b_2\}^*$ such that $v=w_1cw_2$ is a consequence of $R$. Therefore, $w=v_0cw_1cw_2$ is a consequence of $R$. Then, 
by Lemma \ref{cwc=w1c2w2}, we may consider $w'_1,w'_2\in\{a_1,a_2,b_1,b_1\}^*$ such that $cw_1c=w'_1c^2w'_2$ is a consequence of $R$ and so $w=v_0w'_1c^2w'_2w_2$ is a consequence of $R$. 
Since $c^2=(b_1a_1)^{n-3}b_1a_1^3a_2$ is a consequence of $R$, by Lemma \ref{rconseq}, 
being $w'\equiv v_0w'_1 (b_1a_1)^{n-3}b_1a_1^3a_2  w'_2w_2 \in\{a_1,a_2,b_1,b_2\}^*$, 
we obtain that $w=w'$ is a consequence of $R$.

Next, suppose that $|w|_c$ is odd.  Then $|v|_c$ is even and so, by the induction hypothesis, 
there exists $w'\in\{a_1,a_2,b_1,b_2\}^*$ such that $v=w'$ is a consequence of $R$. 
Hence, $w=v_0cw'$ is a consequence of $R$.

\smallskip 

Now, let $w\in\{a_1,a_2,b_1,b_2,c\}^*\setminus\{a_1,a_2,b_1,b_2\}^*$ be any word. 
Then, by taking in account the relation $(R_{11})$, it is clear that 
$$
w= u_0c^{i_1}u_1c^{i_2}u_2c^{i_3}\cdots c^{i_{k-1}}u_{k-1}c^{i_k}u_k
$$ 
is a consequence of $R$, for some $k\in\mathbb{N}$, $i_1,\ldots,i_k\in\{1,2\}$, $u_1,u_2,\ldots,u_{k-1}\in\{a_1,a_2,b_1,b_2\}^+$ and $u_0,u_k\in\{a_1,a_2,b_1,b_2\}^*$.
Notice that $|w|_c$ and $|u_0c^{i_1}u_1c^{i_2}u_2c^{i_3}\cdots c^{i_{k-1}}u_{k-1}c^{i_k}u_k|_c$ have the same parity.

Next, by replacing in $u_0c^{i_1}u_1c^{i_2}u_2c^{i_3}\cdots c^{i_{k-1}}u_{k-1}c^{i_k}u_k$ each $c^2$ 
by $(b_1a_1)^{n-3}b_1a_1^3a_2$, we obtain a word $w'$ such that, 
by Lemma \ref{rconseq}, $u_0c^{i_1}u_1c^{i_2}u_2c^{i_3}\cdots c^{i_{k-1}}u_{k-1}c^{i_k}u_k=w'$ is a consequence of $R$ and 
 the parity of $|u_0c^{i_1}u_1c^{i_2}u_2c^{i_3}\cdots c^{i_{k-1}}u_{k-1}c^{i_k}u_k|_c$ and $|w'|_c$ are the same. 
It follows that $w=w'$ is a consequence of $R$ and $|w|_c$ and $|w'|_c$ have the same parity. 

If $i_1=i_2=\cdots=i_k=2$ (and so $|w|_c$ is even), then $w'\in\{a_1,a_2,b_1,b_2\}^*$, which ends the proof. 

On the other hand, if there exists $j\in\{1,\ldots,k\}$ such that $i_j=1$, then 
$w'$ is a word of the form $v_0cv_1cv_2c\cdots cv_{m-1}cv_m$, for some $m\in\{1,\ldots,k\}$, $v_0,v_m\in\{a_1,a_2,b_1,b_2\}^*$ and $v_1,v_2,\ldots,v_{m-1}\in\{a_1,a_2,b_1,b_2\}^+$. 
Therefore, the result follows by the first part of the proof, as required. 
\end{proof} 

\begin{lemma} \label{w=ascw2}
Let $w\in\{a_1,a_2,b_1,b_2,c\}^*\setminus\{a_1,a_2,b_1,b_2\}^*$ such that $|w|_c$ is odd. Then there exist $w_1\in\{a_1,a_2\}^*$ and $w_2\in\{a_1,a_2,b_1,b_2\}^*$ such that $w=w_1cw_2$ is a consequence of $R$.
\end{lemma}

\begin{proof}
Let us consider $u_1,u_2\in\{a_1,a_2,b_1,b_2\}^*$ such that $w=u_1cu_2$ is a consequence of $R$, which are guaranteed by Lemma \ref{w=ww=w1cw2}. 
Observe that, if $u_1\in\{a_1,a_2\}^*$, then there is nothing left to prove. 
Thus, from now on, we assume that $u_1\in\{a_1,a_2,b_1,b_2\}^*\setminus\{a_1,a_2\}^+$.

We will prove by induction on the length of $u_1$ that there exist 
$w_1\in\{a_1,a_2\}^*$ and $w_2\in\{a_1,a_2,b_1,b_2\}^*$ such that $u_1cu_2=w_1cw_2$ is a consequence of $R$, 
thus completing the proof. 

If $|u_1|=0$ then $u_1\equiv 1 \in \{a_1,a_2\}^*$, and the result follows. 

Let $k\geqslant 1$ and assume that, for all $u,v\in\{a_1,a_2,b_1,b_2\}^*$ such that $|u|<k$, there exist $w_1\in\{a_1,a_2\}^*$ 
and $w_2\in\{a_1,a_2,b_1,b_2\}^*$ such that $ucv=w_1cw_2$ is a consequence of $R$.

Suppose $|u_1|=k$. It follows from relations $(R_1)$, $(R_2)$ and $(R_6)$ that $u_1=u_3$ is a consequence of $R$, for some subword $u_3\in\{a_1,a_2,b_1,b_2\}^*$ of $u_1$ such that none of the words of 
$\{a_1^i \mid i\geqslant n-1\}\cup\{a_2^i  \mid  i\geqslant 2\}\cup\{b_1^i  \mid  i\geqslant 2\}\cup\{b_2^i  \mid  i\geqslant 2\}$ is a factor of $u_3$. 
Notice that $|u_3|\leqslant |u_1|$.

If $|u_3|<|u_1|$ then, by the induction hypothesis, there exist $w_1\in\{a_1,a_2\}^*$ 
and $w_2\in\{a_1,a_2,b_1,b_2\}^*$ 
such that $u_3cu_2=w_1cw_2$ is a consequence of $R$. Since $u_1=u_3$ is a consequence of $R$ then $u_1cu_2=u_3cu_2$ is a consequence of $R$ and so $u_1cu_2=w_1cw_2$ is a consequence of $R$. 

Now, assume that $|u_3|=|u_1|$. Thus, $u_1\equiv u_3$ and, consequently, none of the words of 
$\{a_1^i \mid i\geqslant n-1\}\cup\{a_2^i  \mid  i\geqslant 2\}\cup\{b_1^i  \mid  i\geqslant 2\}\cup\{b_2^i  \mid  i\geqslant 2\}$  is a factor of $u_1$.

Let $u_4\in\{a_1,a_2,b_1,b_2\}^+$ be a suffix of $u_1$ and let $u_5\in\{a_1,a_2,b_1,b_2\}^*$ such that $u_1=u_5u_4$. 
Then, we can choose $u_4$ and $u_5$  satisfying one of the following cases, which we study separately, concluding the proof. 

\smallskip 

Case 1: Assume that $u_4 \equiv b_1$. By Lemma \ref{rconseq}, the relation $b_1c=c$ is a consequence of $R$, which implies that 
$$
u_1cu_2=u_5u_4cu_2=u_5b_1cu_2=u_5cu_2, 
$$
whence $u_1cu_2=u_5cu_2$ is a consequence of $R$. 
Since $|u_5|<|u_1|$ then, by the induction hypothesis, there exist $w_1\in\{a_1,a_2\}^*$  
and $w_2\in\{a_1,a_2,b_1,b_2\}^*$ 
such that $u_5cu_2=w_1cw_2$ and so 
$u_1cu_2=w_1cw_2$ is a consequence of $R$.

\smallskip

Case 2: Let $u_4 \equiv b_2$. Again, by Lemma \ref{rconseq}, 
the relation $b_2c=ca_2(b_1a_1)^{n-3}b_1a_1^3a_2$ is a consequence of $R$. This implies that
$$
u_1cu_2=u_5u_4cu_2=u_5b_2cu_2=u_5ca_2(b_1a_1)^{n-3}b_1a_1^3a_2u_2. 
$$
Thus, $u_1cu_2=u_5ca_2(b_1a_1)^{n-3}b_1a_1^3a_2u_2$ is a consequence of $R$. 
Since $|u_5|<|u_1|$ then, by the induction hypothesis, there exist $w_1\in\{a_1,a_2\}^*$ 
and $w_2\in\{a_1,a_2,b_1,b_2\}^*$ such that 
$u_5ca_2(b_1a_1)^{n-3}b_1a_1^3a_2u_2=w_1cw_2$, whence $u_1cu_2=w_1cw_2$ is a consequence of $R$.

\smallskip

Case 3: Suppose that $u_4 \equiv a_1^{n-2}$. By using the relation $(R_{13})$, we obtain 
$$
u_1cu_2=u_5u_4cu_2=u_5a_1^{n-2}cu_2=u_5a_2cu_2
$$
and so $u_1cu_2=u_5a_2cu_2$ is a consequence of $R$. 
As $|u_5a_2|<|u_1|$, by the induction hypothesis, there exist $w_1\in\{a_1,a_2\}^*$ 
and $w_2\in\{a_1,a_2,b_1,b_2\}^*$ such that 
$u_5a_2cu_2=w_1cw_2$ is a consequence of $R$ and thus we conclude that $u_1cu_2=w_1cw_2$ is a consequence of $R$.

\smallskip

Case 4: Take $u_4 \equiv b_2a_1^j$, for some $j\in\{1,\ldots,n-3\}$. By considering the relations $(R_7)$ and the fact that, 
by Lemma \ref{rconseq}, the relation 
$b_2c=ca_2(b_1a_1)^{n-3}b_1a_1^3a_2$ is a consequence of $R$, we obtain 
$$
u_1cu_2=u_5u_4cu_2=u_5b_2a_1^jcu_2=u_5a_1^jb_2cu_2=u_5a_1^jca_2(b_1a_1)^{n-3}b_1a_1^3a_2u_2, 
$$
whence $u_1cu_2=u_5a_1^jca_2(b_1a_1)^{n-3}b_1a_1^3a_2u_2$ is a consequence of $R$. 
Since $|u_5a_1^j|<|u_1|$ then, by the induction hypothesis, there exist $w_1\in\{a_1,a_2\}^*$ and $w_2\in\{a_1,a_2,b_1,b_2\}^*$ such that $u_5a_1^jca_2(b_1a_1)^{n-3}b_1a_1^3a_2u_2=w_1cw_2$ is a consequence of $R$. 
Therefore, $u_1cu_2=w_1cw_2$ is a consequence of $R$.

\smallskip 

Case 5: Assume that $u_4 \equiv b_1a_1$. Then, by applying relation $(R_{15})$, we have 
$$
u_1cu_2=u_5u_4cu_2=u_5b_1a_1cu_2=u_5cb_2u_2 
$$
and so $u_1cu_2=u_5cb_2u_2$ is a consequence of $R$. 
Since $|u_5|<|u_1|$ then, by the induction hypothesis, there exist $w_1\in\{a_1,a_2\}^*$ and $w_2\in\{a_1,a_2,b_1,b_2\}^*$ such that $u_5cb_2u_2=w_1cw_2$ is a consequence of $R$. Thus, $u_1cu_2=w_1cw_2$ is a consequence of $R$.

\smallskip

Case 6: Let us consider that $u_4 \equiv b_1a_1^j$, for some $j\in\{2,\ldots,n-3\}$. 
By the relations $(R_{16})$, we get 
$$
u_1cu_2=u_5u_4cu_2=u_5b_1a_1^jcu_2=u_5a_1^jcu_2,
$$
whence $u_1cu_2=u_5a_1^jcu_2$ is a consequence of $R$. 
As $|u_5a_1^j|<|u_1|$, we can use the induction hypothesis, which guarantees that there exist 
$w_1\in\{a_1,a_2\}^*$ and $w_2\in\{a_1,a_2,b_1,b_2\}^*$ such that $u_5a_1^jcu_2=w_1cw_2$ is a consequence of $R$. 
It follows that $u_1cu_2=w_1cw_2$ is a consequence of $R$.

\smallskip

Case 7: Suppose that $u_4 \equiv a_2a_1^j$, for some $j\in\{1,\ldots,n-3\}$. 
Then, by the relations $(R_{14})$, we obtain
$$
u_1cu_2=u_5u_4cu_2=u_5a_2a_1^jcu_2=u_5a_1^jcu_2,
$$
which implies that $u_1cu_2=u_5a_1^jcu_2$ is a consequence of $R$. 
Clearly, $|u_5a_1^j|<|u_1|$ and so, by the induction hypothesis, there exist $w_1\in\{a_1,a_2\}^*$ and $w_2\in\{a_1,a_2,b_1,b_2\}^*$ such that $u_5a_1^jcu_2=w_1cw_2$ is a consequence of $R$. Hence, $u_1cu_2=w_1cw_2$ is a consequence of $R$.

\smallskip

Case 8: Let $u_4 \equiv b_1a_2$. We have that $b_1c=c$ is a consequence of $R$, 
by Lemma \ref{rconseq}. By considering also the relations $(R_7)$, we have 
$$
u_1cu_2=u_5u_4cu_2=u_5b_1a_2cu_2=u_5a_2b_1cu_2=u_5a_2cu_2
$$
and so $u_1cu_2=u_5a_2cu_2$ is a consequence of $R$. 
Since $|u_5a_2|<|u_1|$ then, by the induction hypothesis, 
there exist $w_1\in\{a_1,a_2\}^*$ and $w_2\in\{a_1,a_2,b_1,b_2\}^*$ such that $u_5a_2cu_2=w_1cw_2$ is a consequence of $R$, 
whence $u_1cu_2=w_1cw_2$ is a consequence of $R$.

\smallskip

Case 9: Take $u_4 \equiv b_2a_2$. By Lemma \ref{rconseq}, we have that the relation 
$b_2c=ca_2(b_1a_1)^{n-3}b_1a_1^3a_2$ is a consequence of $R$. 
By applying also relations $(R_7)$, we get
$$
u_1cu_2=u_5u_4cu_2=u_5b_2a_2cu_2=u_5a_2b_2cu_2=u_5a_2ca_2(b_1a_1)^{n-3}b_1a_1^3a_2u_2.
$$
So $u_1cu_2=u_5a_2ca_2(b_1a_1)^{n-3}b_1a_1^3a_2u_2$ is a consequence of $R$.
It is clear that $|u_5a_2|<|u_1|$ and thus, by the induction hypothesis, 
there exist $w_1\in\{a_1,a_2\}^*$ and $w_2\in\{a_1,a_2,b_1,b_2\}^*$ such that 
$u_5a_2ca_2(b_1a_1)^{n-3}b_1a_1^3a_2u_2=w_1cw_2$ is a consequence of $R$. Therefore, 
$u_1cu_2=w_1cw_2$ is a consequence of $R$.

\smallskip

Case 10: Finally, assume that $u_4 \equiv a_1a_2$. Since $a_1a_2c=c$ is a consequence of $R$, 
by Lemma \ref{rconseq}, we have 
$$
u_1cu_2=u_5u_4cu_2=u_5a_1a_2cu_2=u_5cu_2,
$$
whence $u_1cu_2=u_5cu_2$ is a consequence of $R$. 
Since $|u_5|<|u_1|$, the induction hypothesis guarantees that there exist $w_1\in\{a_1,a_2\}^*$ and $w_2\in\{a_1,a_2,b_1,b_2\}^*$ such that $u_5cu_2=w_1cw_2$ is a consequence of $R$. Thus, $u_1cu_2=w_1cw_2$ is a consequence of $R$, as required. 
\end{proof}

\begin{lemma} \label{u3cvarphi}
Let $u_1,v_1\in\{a_1,a_2\}^*$ and $u_2,v_2\in\{a_1,a_2,b_1,b_2\}^*$ be such that $(u_1cu_2)\varphi=(v_1cv_2)\varphi$. Then 
$(cu_2)\varphi=(cv_2)\varphi$ and, there exist $u_3,v_3\in\{a_1,a_2,b_1\}^*$ such that $u_1cu_2=u_3cu_2$ and $v_1cv_2=v_3cv_2$ are consequences of $R$ and $(u_3c)\varphi=(v_3c)\varphi$.
\end{lemma}
\begin{proof}
First, observe that 
$$
(u_1cu_2)\varphi=(u_1\varphi)(c\varphi)(u_2\varphi)=(u_1\varphi)\gamma(u_2\varphi)
\quad\text{and}\quad 
(v_1cv_2)\varphi=(v_1\varphi)(c\varphi)(v_2\varphi)=(v_1\varphi)\gamma(v_2\varphi)
$$
and so, as $|\dom(\gamma)|=2$, we have $|\dom((u_1cu_2)\varphi)|=|\dom((v_1cv_2)\varphi)|\leqslant 2$. 
Since $u_1,v_1\in\{a_1,a_2\}^*$,  it follows that  
$u_1\varphi,v_1\varphi\in\langle\alpha_1,\alpha_2\rangle=\{\alpha\in\DPS_n \mid  |\dom(\alpha)|=n\}$, whence 
$\dom(u_1\varphi)=\im(u_1\varphi)=\dom(v_1\varphi)=\im(v_1\varphi)=\{0,1,\ldots,n-1\}$ and $0(u_1\varphi)=0(v_1\varphi)=0$.
Moreover, 
$$
\dom((u_1cu_2)\varphi)\subseteq\dom((u_1c)\varphi)=\dom((u_1\varphi)\gamma)=\{0,i\}\quad\text{and}\quad  
(u_1\varphi)\gamma=\begin{pmatrix}0&i\\1&0\end{pmatrix},
$$ 
for some $i\in\{1,\ldots,n-1\}$, and 
$$
\dom((v_1cv_2)\varphi)\subseteq\dom((v_1c)\varphi)=\dom((v_1\varphi)\gamma)=\{0,j\}\quad\text{and}\quad  
(v_1\varphi)\gamma=\begin{pmatrix}0&j\\1&0\end{pmatrix},
$$ 
for some $j\in\{1,\ldots,n-1\}$. We also have  
$$
\dom((cu_2)\varphi)=\dom(\gamma(u_2\varphi))\subseteq \{0,1\}
\quad\text{and}\quad
\dom((cv_2)\varphi)=\dom(\gamma(v_2\varphi))\subseteq \{0,1\},
$$
$$
1\in\dom((cu_2)\varphi) \Longleftrightarrow i=j \text{ and } i\in \dom((u_1cu_2)\varphi)=\dom((v_1cv_2)\varphi) 
\Longleftrightarrow 1\in\dom((cv_2)\varphi)
$$
and 
$$
0\in\dom((cu_2)\varphi) \Longleftrightarrow 0 \in \dom((u_1cu_2)\varphi)=\dom((v_1cv_2)\varphi) 
\Longleftrightarrow 0\in\dom((cv_2)\varphi),
$$
whence $\dom((cu_2)\varphi)=\dom((cv_2)\varphi)$. 
Furthermore, 
if $0\in\dom((cu_2)\varphi)$ then 
$$
0(cu_2)\varphi=(0(u_1\varphi))(cu_2)\varphi=0(u_1cu_2)\varphi = 
0(v_1cv_2)\varphi =(0(v_1\varphi))(cv_2)\varphi=0(cv_2)\varphi
$$
and if $1\in\dom((cu_2)\varphi)$ then 
\begin{multline}\nonumber
1(cu_2)\varphi=1\gamma(u_2\varphi)= 0(u_2\varphi)=(i(u_1\varphi)\gamma)(u_2\varphi)= 
i(u_1cu_2)\varphi = \\
= i(v_1cv_2)\varphi = 
(i(v_1\varphi)\gamma)(v_2\varphi)=0(v_2\varphi)=1\gamma(v_2\varphi)=1(cv_2)\varphi.  
\end{multline}
Thus $(cu_2)\varphi=(cv_2)\varphi$. 

\smallskip 

Next, we divide the proof in two cases.  

\smallskip 

Case 1: Suppose that $|\dom((u_1cu_2)\varphi)|=2$ 
or $\dom((u_1cu_2)\varphi)=\{i\}$.  
Then, 
$$
\{0,i\}=\dom((u_1cu_2)\varphi)=\dom((v_1cv_2)\varphi)=\{0,j\} 
$$
or  
$$
\{i\}=\dom((u_1cu_2)\varphi)=\dom((v_1cv_2)\varphi)\subseteq \{0,j\}
$$ 
and so, in both scenarios, we have 
$j=i$ and so $(u_1c)\varphi=(u_1\varphi)\gamma=(v_1\varphi)\gamma=(v_1c)\varphi$. 
Hence, being $u_3\equiv u_1$ and $v_3\equiv v_1$, we have $u_1cu_2\equiv u_3cu_2$, 
$v_1cv_2\equiv v_3cv_2$ and $(u_3c)\varphi=(v_3c)\varphi$, which concludes the proof in this case. 

\smallskip 

Case 2: Now, admit that $\dom((u_1cu_2)\varphi)=\{0\}$ 
or $\dom((u_1cu_2)\varphi)=\emptyset$. 
In both scenarios, it follows that $i\not\in \dom((u_1cu_2)\varphi)=\dom((v_1cv_2)\varphi)$. 
Then, by the above observations, we conclude that $1\not\in \dom((cu_2)\varphi)=\dom((cv_2)\varphi)$ 
and so $0\not\in\dom(u_2\varphi)\cup\dom(v_2\varphi)$. 
Hence, $u_2\varphi, v_2\varphi \not\in\langle\alpha_1,\alpha_2,\beta_1\rangle$ and so 
$|u_2|_{b_2}\geqslant 1$ and $|v_2|_{b_2}\geqslant 1$. 
Therefore, taking in account relations $(R_6)$ and $(R_7)$, we deduce that 
$u_2=b_2u_2$ and $v_2=b_2v_2$ are consequences of $R$ and so, by applying also the relation $(R_{15})$, 
we conclude that $u_1cu_2=u_1b_1a_1cu_2$ and $v_1cv_2=v_1b_1a_1cv_2$ are consequences of $R$.

Now, since
$$
(b_1a_1c)\varphi=(b_1\varphi)(a_1\varphi)(c\varphi)=\beta_1\alpha_1\gamma=
\begin{pmatrix}0\\1\end{pmatrix}
$$
and $0(u_1\varphi)=0(v_1\varphi)=0$, then
$$
(u_1b_1a_1c)\varphi=(u_1\varphi)(b_1a_1c)\varphi=\begin{pmatrix}0\\1\end{pmatrix}=(v_1\varphi)(b_1a_1c)\varphi=(v_1b_1a_1c)\varphi 
$$
and so, by considering $u_3\equiv u_1b_1a_1\in\{a_1,a_2,b_1\}^*$ and $v_3\equiv v_1b_1a_1\in\{a_1,a_2,b_1\}^*$, 
we obtain that $u_1cu_2=u_3cu_2$ and $v_1cv_2=v_3cv_2$ are consequences of $R$ and $(u_3c)\varphi=(v_3c)\varphi$, as required. 
\end{proof}

\begin{lemma}\label{samepar}
Let $w_1\in\{a_1,a_2,b_1,b_2,c\}^*\setminus\{a_1,a_2,b_1,b_2\}^*$ and $w_2\in\{a_1,a_2,b_1,b_2,c\}^*$ be such that $|w_1|_c$ and $|w_2|_c$ have the same parity. If $w_1\varphi=w_2\varphi$ then $w_1=w_2$ is a consequence of $R$.
\end{lemma}
\begin{proof}
First, suppose that $w_2\in\{a_1,a_2,b_1,b_2\}^*$. Then $|w_2|_c$ is even and so $|w_1|_c$ is also even. Thus, 
by Lemma \ref{w=ww=w1cw2}, there exists $w\in\{a_1,a_2,b_1,b_2\}^*$ such that $w_1=w$ is a consequence of $R$. Hence $w_2\varphi=w_1\varphi=w\varphi$, which implies, by Lemma \ref{a1a2b1b2}, that $w_2=w$ is a consequence of $R$. 
Therefore $w_1=w_2$ is a consequence of $R$.

Now, admit that $w_2\in\{a_1,a_2,b_1,b_2,c\}^*\setminus\{a_1,a_2,b_1,b_2\}^*$.

If $|w_1|_c$ and $|w_2|_c$ are both even then, by Lemma \ref{w=ww=w1cw2}, there exist $u_1,u_2\in\{a_1,a_2,b_1,b_2\}^*$ such that $w_1=u_1$ and $w_2=u_2$ are consequences of $R$. This implies that $u_1\varphi=w_1\varphi=w_2\varphi=u_2\varphi$. 
Therefore, by Lemma \ref{a1a2b1b2}, $u_1=u_2$ is a consequence of $R$ and thus $w_1=w_2$ is a consequence of $R$.

If $|w_1|_c$ and $|w_2|_c$ are both odd then, by Lemma \ref{w=ascw2}, there exist $u_1,v_1\in\{a_1,a_2\}^*$ and $u_2,v_2\in\{a_1,a_2,b_1,b_2\}^*$ such that $w_1=u_1cu_2$ and $w_2=v_1cv_2$ are consequences of $R$. 
Thus, $(u_1cu_2)\varphi=w_1\varphi=w_2\varphi=(v_1cv_2)\varphi$ and so, 
by Lemma \ref{u3cvarphi}, $(cu_2)\varphi=(cv_2)\varphi$ and there exist $u_3,v_3\in\{a_1,a_2,b_1\}^*$ 
such that $u_1cu_2=u_3cu_2$ and $v_1cv_2=v_3cv_2$ are consequences of $R$ and $(u_3c)\varphi=(v_3c)\varphi$.
Hence, we have 
$$
(c^2u_2)\varphi=(c\varphi)((cu_2)\varphi)=(c\varphi)((cv_2)\varphi)=(c^2v_2)\varphi
\quad\text{and}\quad 
(u_3c^2)\varphi=((u_3c)\varphi)(c\varphi)=((v_3c)\varphi)(c\varphi)=(v_3c^2)\varphi.
$$
Now, since $|c^2u_2|_c=|c^2v_2|_c=2$ and $|u_3c^2|_c=|v_3c^2|=2$ then, by the first part of the proof, 
we conclude that $c^2u_2=c^2v_2$ and $u_3c^2=v_3c^2$ are consequences of $R$, 
whence $c^3u_2=c^3v_2$ and $u_3c^3=v_3c^3$ are consequences of $R$ and so, by relation $(R_{11})$, 
$cu_2=cv_2$ and $u_3c=v_3c$ are also consequences of $R$. Thus, we have
$$
w_1=u_1cu_2=u_3cu_2=u_3cv_2=v_3cv_2=v_1cv_2=w_2,
$$
which implies $w_1=w_2$ is a consequence of $R$, as required. 
\end{proof}

Finally, we present our last lemma. 

\begin{lemma}\label{difpar}
Let $w_1\in\{a_1,a_2,b_1,b_2,c\}^*\setminus\{a_1,a_2,b_1,b_2\}^*$ and $w_2\in\{a_1,a_2,b_1,b_2,c\}^*$ be such that $|w_1|_c$ and $|w_2|_c$ have different parity. If $w_1\varphi=w_2\varphi$ then $w_1=w_2$ is a consequence of $R$.
\end{lemma}

\begin{proof}
Observe that, since $|w|_c=0$, so an even number, for all $w\in\{a_1,a_2,b_1,b_2\}^*$, we may 
suppose, without loss of generality, that $|w_1|_c$ is odd and $|w_2|_c$ is even.

If $w_2\in\{a_1,a_2,b_1,b_2\}^*$ then it is obvious that $w_2=v$ is a consequence of $R$, where $v\equiv w_2$.

If $w_2\in\{a_1,a_2,b_1,b_2,c\}^*\setminus\{a_1,a_2,b_1,b_2\}^*$ then, as $|w_2|_c$ is even, 
Lemma \ref{w=ww=w1cw2} guarantees us the existence of $v\in\{a_1,a_2,b_1,b_2\}^*$ such that $w_2=v$ is a consequence of $R$.

Either way, there exists $v\in\{a_1,a_2,b_1,b_2\}^*$ such that $w_2=v$ is a consequence of $R$.

Also, since $|w_1|_c$ is odd, by Lemma \ref{w=ascw2} there exist $u_1\in\{a_1,a_2\}^*$ and $u_2\in\{a_1,a_2,b_1,b_2\}^*$ such that $w_1=u_1cu_2$ is a consequence of $R$.

Next, as in the proof of the Lemma \ref{u3cvarphi}, we have 
$$
\dom((u_1cu_2)\varphi)\subseteq\dom((u_1c)\varphi)=\dom((u_1\varphi)\gamma)=\{0,i\}\quad\text{and}\quad  
(u_1\varphi)\gamma=\begin{pmatrix}0&i\\1&0\end{pmatrix},
$$ 
for some $i\in\{1,\ldots,n-1\}$, 
and $(u_1cu_2)\varphi=w_1\varphi=w_2\varphi=v\varphi$.
On the other hand, since $u_2,v\in\{a_1,a_2,b_1,b_2\}^*$, we deduce that 
$$
0\in\dom(u_2\varphi)\Longrightarrow 0(u_2\varphi)=0, \quad 
1\in\dom(u_2\varphi)\Longrightarrow 1(u_2\varphi)\neq 0,
$$
$$
0\in\dom(v\varphi)\Longrightarrow 0(v\varphi)=0\quad\text{and}\quad  
i\in\dom(v\varphi)\Longrightarrow i(v\varphi)\neq 0, 
$$ 
whence 
$$
0\in\dom(v\varphi) \Longrightarrow  1\in\dom(u_2\varphi) ~\text{and}~0=0(v\varphi)=0((u_1cu_2)\varphi)=1(u_2\varphi)\neq 0
$$
and
$$
i\in\dom(v\varphi) \Longrightarrow  0\in\dom(u_2\varphi) ~\text{and}~0=0(u_2\varphi)=i((u_1cu_2)\varphi)=i(v\varphi)\neq 0,
$$
from which we conclude that $(u_1cu_2)\varphi=v\varphi=\emptyset$. 

Now, since $|b_2cb_2|_c=|u_1cu_2|_c=1$ and
$
(b_2cb_2)\varphi=(b_2\varphi)(c\varphi)(b_2\varphi)=\beta_2\gamma\beta_2=\emptyset=(u_1cu_2)\varphi
$, 
by Lemma \ref{samepar}, we get that $u_1cu_2=b_2cb_2$ is a consequence of $R$. 
Moreover, we also have that $|b_2cb_2c|_c$ and $|v|_c$ are both even and
$(b_2cb_2c)\varphi=((b_2cb_2)\varphi)(c\varphi)=\emptyset\gamma=\emptyset=v\varphi$ and so, 
by Lemma \ref{samepar}, we also obtain that $v=b_2cb_2c$ is a consequence of $R$.

Therefore, by the relation $(R_{19})$ we deduce that $b_2cb_2c=b_2cb_2$ is a consequence of $R$, 
whence $u_1cu_2=v$ is a consequence of $R$ and so $w_1=w_2$ is a consequence of $R$, as required. 
\end{proof}

Finally, as a consequence of Proposition \ref{provingpresentation} and Lemmas \ref{genrel}, \ref{a1a2b1b2}, \ref{samepar} 
and \ref{difpar}, we immediately have our main result of this section: 

\begin{theorem}
For $n\geqslant 4$, the monoid $\DPS_n$ is defined by the presentation $\langle A\mid R \rangle$ on $5$ generators and $3n+9$ relations.
\end{theorem}

\medskip 

For completeness, we end this section, and the paper, with the following presentations for the monoids $\DPS_1$, $\DPS_2$ and $\DPS_3$. 

\smallskip 

Since $\DPS_1= \I(\{0\})=\langle\emptyset\rangle$, it is obvious that $\langle z\mid z^2=z\rangle$ is a presentation for  $\DPS_1$. 

\smallskip 

Next, as $\DPS_2= \I(\{0,1\})=\left\langle\begin{pmatrix}0&1\\1&0\end{pmatrix},\begin{pmatrix}0\\0\end{pmatrix}\right\rangle$, 
it is easy to check that 
$$
\left\langle a,s\mid a^2=1,\: s^2=s,\: (sa)^2=sas=(as)^2\right\rangle
$$ 
is a presentation for $\DPS_2$ associated to this set of generators. 

\smallskip 

Finally, recall that $\DPS_3=\langle\alpha_1,\beta_2,\gamma\rangle$. Then, by using GAP computational system \cite{GAP4}, we can easily verify that 
$$
\left\langle a_1,b_2,c\mid a_1^2=1,\: b_2^2=b_2,\: a_1b_2=b_2a_1,\: c^3=c,\: b_2c^2=c^2b_2=ca_1c,\: (a_1c^2)^2=(c^2a_1)^2,\: 
(b_2c)^2=b_2cb_2\right\rangle
$$
is a presentation for $\DPS_3$ associated to the generators $\alpha_1$, $\beta_2$ and $\gamma$.

\bigskip 

\lastpage 


\begin{thebibliography}{00} 

\bibitem{Aizenstat:1958}
A.Ya. A\u{\i}zen\v{s}tat, 
Defining relations of finite symmetric semigroups, 
Mat. Sb. N. S.  45 (1958), 261--280 (Russian). 

\bibitem{Aizenstat:1962}
A.Ya. A\u{\i}zen\v{s}tat, 
The defining relations of the endomorphism semigroup of a finite linearly ordered set, 
Sibirsk. Mat. 3 (1962), 161--169 (Russian). 

\bibitem{AlKharousi&Kehinde&Umar:2014} 
F. Al-Kharousi, R. Kehinde and A. Umar,
Combinatorial results for certain semigroups of partial isometries of a finite chain, 
Australas. J. Combin. 58 (2014), 365--375.

\bibitem{AlKharousi&Kehinde&Umar:2016} 
F. Al-Kharousi, R. Kehinde and A. Umar,
On the semigroup of partial isometries of a finite chain, 
Communications in Algebra 44 (2016), 639--647. 

\bibitem{Araujo&al:2015}
J. Ara\'ujo,  W. Bentz, J.D. Mitchell and C. Schneider,
The rank of the semigroup of transformations stabilising a partition of a finite set,
Mathematical Proceedings of the Cambridge Philosophical Society,
159  (2015), 339--353.

\bibitem{Cicalo&al:2015}
S. Cical\`o, V.H. Fernandes and C. Schneider, 
Partial transformation monoids preserving a uniform partition, 
Semigroup Forum 90  (2015), 532--544. 

\bibitem{East:2011}
J. East, 
Generators and relations for partition monoids and algebras, 
J. Algebra 339 (2011), 1--26.

\bibitem{Feng&al:2019}
Y.-Y. Feng, A. Al-Aadhami, I. Dolinka, J. East and V. Gould, 
Presentations for singular wreath products, 
J. Pure Appl. Algebra 223 (2019), 5106--5146. 

\bibitem{Fernandes:2001}
V.H. Fernandes, 
The monoid of all injective order preserving partial transformations on a finite chain, 
Semigroup Forum 62 (2001), 178-204. 

\bibitem{Fernandes:2002survey}
V.H. Fernandes, 
Presentations for some monoids of partial transformations on a finite chain: a survey, 
Semigroups, Algorithms, Automata and Languages, 
eds. Gracinda M. S. Gomes \& Jean-\'Eric Pin \& Pedro V. Silva, 
World Scientific (2002), 363--378.

\bibitem{Fernandes&Gomes&Jesus:2004}
V.H. Fernandes, G.M.S. Gomes and M.M. Jesus, 
Presentations for some monoids of injective partial transformations on a finite chain, 
Southeast Asian Bull. Math. 28 (2004), 903--918.

\bibitem{Fernandes&al:2014}
V.H. Fernandes, P. Honyam, T.M. Quinteiro and B. Singha,
On semigroups of endomorphisms of a chain with restricted range,
Semigroup Forum 89 (2014), 77--104. 

\bibitem{Fernandes&al:2019}
V.H. Fernandes, J. Koppitz and T. Musunthia, 
The rank of the semigroup of all order-preserving transformations on a finite fence, 
Bulletin of the Malaysian Mathematical Sciences Society  
42 (2019), 2191--2211.
 
\bibitem{Fernandes&Quinteiro:2014}
V.H. Fernandes and T.M. Quinteiro,
On the ranks of certain monoids of transformations that preserve a uniform partition,
Communications in Algebra 42 (2014), 615--636.

\bibitem{Fernandes&Quinteiro:2016}
V.H. Fernandes and T.M. Quinteiro,
Presentations for monoids of finite partial isometries,
Semigroup Forum 93 (2016), 97--110.

\bibitem{Fernandes&Sanwong:2014}
V.H. Fernandes and J. Sanwong,
On the rank of semigroups of transformations on a finite set with restricted range,
Algebra Colloquium 21 (2014), 497--510.

\bibitem{GAP4}
  The GAP~Group, \emph{GAP -- Groups, Algorithms, and Programming, 
  Version 4.11.1}; 2021. \newline (https://www.gap-system.org)

\bibitem{Gomes&Howie:1992}
G.M.S. Gomes and J.M. Howie, 
On the ranks of certain semigroups of order-preserving transformations, 
Semigroup Forum 45 (1992), 272--282.

\bibitem{Howie:1995} 
J.M. Howie, 
Fundamentals of Semigroup Theory, 
Oxford, Oxford University Press, 1995.

\bibitem{Howie&Ruskuc:1995} 
J.M. Howie and N. Ru\v{s}kuc,  
Constructions and presentations for monoids,
Communications in Algebra 22 (1994), 6209--6224.

\bibitem{Lallement:1979}
G. Lallement,
Semigroups and Combinatorial Applications,
John Wiley \& Sons, New York, 1979.

\bibitem{Moore:1897}
E.H. Moore, 
Concerning the abstract groups of order $k!$ and 
	$\frac12k!$ holohedrically isomorphic with the symmetric and 
	the alternating substitution groups on $k$ letters, 
Proc. London Math. Soc. 28 (1897), 357--366. 

\bibitem{Popova:1961}
L.M. Popova, 
The defining relations of certain semigroups of partial transformations of a finite set, 
Leningrad. Gos. Ped. Inst. U\v cen. Zap. 218 (1961), 191--212 (Russian).

\bibitem{Popova:1962}
L.M. Popova, 
Defining relations of a semigroup of partial endomorphisms of a finite linearly ordered set,
Leningrad. Gos. Ped. Inst. U\v cen. Zap. 238 (1962), 78--88 (Russian). 

\bibitem{Ruskuc:1995}
N. Ru\v{s}kuc, 
Semigroup Presentations, 
Ph. D. Thesis, University of St-Andrews, 1995.

\end{thebibliography}
\end{document}